\documentclass[aap, preprint]{imsart}

\usepackage[OT1]{fontenc}

\usepackage[authoryear]{natbib}
\usepackage[makeroom]{cancel}
\usepackage{amsthm, amsmath, amssymb, enumerate, mathrsfs, mathtools, upgreek}  
\RequirePackage[colorlinks,citecolor=blue,urlcolor=blue,breaklinks]{hyperref} 

\startlocaldefs
\theoremstyle{plain}
\newtheorem{theorem}{Theorem}[section]
\newtheorem{proposition}{Proposition}[section]
\newtheorem{lemma}{Lemma}[section]

\newtheorem{remark}{Remark}[section]
\newtheorem{definition}{Definition}[section]

\newcommand{\eref}[1]{\eqref{#1}}

\newcommand{\sref}[1]{Section~\ref{#1}}
\newcommand{\tref}[1]{Table~\ref{#1}}
\newcommand{\fref}[1]{Figure~\ref{#1}}
\newcommand{\cref}[1]{Chapter~\ref{#1}}

\newcommand{\lemmaref}[1]{Lemma~\ref{#1}}
\newcommand{\propref}[1]{Proposition~\ref{#1}}
\newcommand{\thmref}[1]{Theorem~\ref{#1}}

\def\sG{\mathscr{G}}
\def\sH{\mathscr{H}}

\def\sL{\mathscr{L}}

\def\sS{\mathscr{S}}

\def\bbC{\mathbb{C}}

\def\bbE{\mathbb{E}}

\def\bbI{\mathbb{I}}

\def\bbN{\mathbb{N}}

\def\bbP{\mathbb{P}}

\def\bbR{\mathbb{R}}

\def\cA{{\mathcal A}}

\def\cC{{\mathcal C}}
\def\cD{{\mathcal D}}
\def\cE{{\mathcal E}}

\def\cN{{\mathcal N}}

\def\cP{{\mathcal P}}
\def\cQ{{\mathcal Q}}

\def\cS{{\mathcal S}}

\def\cY{{\mathcal Y}}

\def\ga{\alpha}
\def\gb{\beta}

\def\gd{\delta}

\def\gl{\lambda}

\def\go{\omega}

\def\gq{\theta}
\def\gr{\rho}
\def\gs{\sigma}

\def\gx{\xi}

\def\gD{\Delta}

\def\gL{\Lambda}

\def\bfmath#1{\boldsymbol{#1}}

\def\bfa{{\bfmath{a}}}
\def\bfb{{\bfmath{b}}}
\def\bfc{{\bfmath{c}}}
\def\bfd{{\bfmath{d}}}
\def\bfe{{\bfmath{e}}}

\def\bfl{{\bfmath{l}}}
\def\bfm{{\bfmath{m}}}
\def\bfn{{\bfmath{n}}}

\def\bfr{{\bfmath{r}}}
\def\bfs{{\bfmath{s}}}

\def\bfv{{\bfmath{v}}}
\def\bfw{{\bfmath{w}}}
\def\bfx{{\bfmath{x}}}
\def\bfy{{\bfmath{y}}}
\def\bfz{{\bfmath{z}}}

\def\bfD{\bfmath{D}}

\def\bfG{\bfmath{G}}

\def\bfI{\bfmath{I}}

\def\bfM{\bfmath{M}}

\def\bfP{\bfmath{P}}

\def\bfR{\bfmath{R}}

\def\bfU{\bfmath{U}}

\def\bfW{\bfmath{W}}
\def\bfX{\bfmath{X}}
\def\bfY{\bfmath{Y}}
\def\bfZ{\bfmath{Z}}

\def\ds{\displaystyle}
\def\A{A}
\def\B{B}
\def\G{G}
\def\bfXN{\bfX^{(N)}}
\def\bfYN{\bfY^{(N)}}
\def\bfDN{\bfD^{(N)}}
\def\bfmu{\bfmath{\mu}}
\def\bfnu{\bfmath{\nu}}

\def\bfgs{\bfmath{\gs}}
\def\bfxi{\bfmath{\gx}}
\def\bfzero{\bfmath{0}}
\def\eApprox{\mathring{e}}
\def\eG#1{\mathring{e}_{\text{\rm\tiny Gaussian}}^{(#1)}}
\def\eT#1{\mathring{e}_{\text{\rm\tiny True}}^{(#1)}}
\def\partition#1{\cP_{#1}}
\def\partitiontwo#1{\cQ_{#1}}
\def\haplist#1{\bfl^{(#1)}}
\def\haplistA#1{\bfl^{(#1)_\A}}
\def\haplistB#1{\bfl^{(#1)_\B}}
\def\phammer#1#2{(#1)_{#2\uparrow}}
\def\si{\sum_{i=1}^K}
\def\sk{\sum_{k=1}^K}
\def\sj{\sum_{j=1}^L}
\def\sl{\sum_{l=1}^L}
\def\tA{\theta_\A}
\def\tB{\theta_\B}
\def\ut{\uptau}
\def\Ma{\bfR} 

\def\i{{\sc I}}
\def\ii{{\sc II}}
\def\iii{{\sc III}}
\def\iv{{\sc IV}}
\def\v{{\sc V}}
\def\vi{{\sc VI}}
\def\vii{{\sc VII}}

\def\cov{\bbC\textrm{ov}}

\def\dd#1{\text{\rm d}#1} 

\makeatletter
\newcommand{\specialcell}[1]{\ifmeasuring@#1\else\omit$\displaystyle#1$\ignorespaces\fi} 
\makeatother

\endlocaldefs
\begin{document}
\begin{frontmatter}

\title{Tractable diffusion and coalescent processes for weakly correlated loci}
\runtitle{Evolution of weakly correlated loci} 

\author{\fnms{Paul A.} \snm{Jenkins}\ead[label=e1]{p.jenkins@warwick.ac.uk}\thanksref{t1}}
\thankstext{t1}{Supported in part by EPSRC Research Grant EP/L018497/1 and an NIH Grant R01-GM094402.}
 \address{Department of Statistics\\ University of Warwick\\ Coventry CV4 7AL\\ UK\\ \printead{e1}}
\affiliation{University of Warwick}
\and
\author{\fnms{Paul} \snm{Fearnhead}\ead[label=e2]{p.fearnhead@lancaster.ac.uk}}
 \address{Department of Mathematics and Statistics\\ Lancaster University\\ Lancaster
LA1 4YF\\ UK\\ \printead{e2}}
\affiliation{Lancaster University}
\and
\author{\fnms{Yun S.} \snm{Song}\corref{}\ead[label=e3]{yss@stat.berkeley.edu}\thanksref{t2}}
\thankstext{t2}{Supported in part by an NIH Grant R01-GM094402, and a Packard Fellowship for Science and Engineering.}
\address{Department of Statistics and\\ ~Computer Science Division\\ University of California, Berkeley\\ Berkeley, CA 94720\\ USA\\\printead{e3}}
\affiliation{University of California, Berkeley}

\runauthor{P. A. Jenkins, P. Fearnhead \& Y. S. Song}

\begin{abstract}
Widely used models in genetics include the Wright-Fisher diffusion and its moment dual, Kingman's coalescent. Each has a multilocus extension but under neither extension is the sampling distribution available in closed-form, and their computation is extremely difficult. In this paper we \emph{derive} two new multilocus population genetic models, one a diffusion and the other a coalescent process, which are much simpler than the standard models, but which capture their key properties for large recombination rates. The diffusion model is based on a central limit theorem for density dependent population processes, and we show that the sampling distribution is a linear combination of moments of Gaussian distributions and hence available in closed-form. The coalescent process is based on a probabilistic coupling of the ancestral recombination graph to a simpler genealogical process which exposes the leading dynamics of the former. We further demonstrate that when we consider the sampling distribution as an asymptotic expansion in inverse powers of the recombination parameter, the sampling distributions of the new models agree with the standard ones up to the first two orders.
\end{abstract}



\end{frontmatter}

\section{Introduction}
\label{sec:intro}
The basis of many important problems in genetics is to find an expression for a sampling distribution or likelihood. Valuable tools in this endeavour are stochastic models of allele frequency evolution forwards in time, and their dual genealogical processes backwards in time. In particular, the numerous variants of the Wright-Fisher diffusion and Kingman's coalescent, respectively, have focused attention on the scaling limit as the population size goes to infinity, leading from a (complicated) finite-population model of reproduction to a (simpler) infinite-population limit. At a single genetic locus, the problem of computing sampling distributions in these models is well studied, with even some closed-form formulas available \citep{wri:1949, ewe:1972,jen:son:2011,bha:etal:2012}. However, with ongoing technological developments in high-throughput DNA sequencing, large genomic datasets are becoming available and it is necessary to consider multilocus models. Inter-locus recombination quickly makes such models intractable; for neither the Wright-Fisher diffusion with recombination nor the coalescent with recombination---or \emph{ancestral recombination graph} (ARG)---
is it possible to obtain a closed-form expression for the sampling distribution. This has remained a notoriously difficult problem, and to make progress using these models it has usually been necessary to resort to computationally-intensive techniques such as importance sampling \citep{gri:mar:1996, fea:don:2001, gri:etal:2008, jen:gri:2011}, Markov chain Monte Carlo \citep{kuh:etal:2000, nie:2000, wan:ran:2008, ras:etal:2014}, or other numerical approximations 
\citep{boi:loi:2007, miu:2011}. Denoting the population-scaled recombination parameter by $\gr$, only in the special cases of $\gr = 0$ or $\gr = \infty$ is it possible to make progress analytically, since then we are back to a single locus, or to many independent single loci, respectively.

In another direction, we 
have considered an analytic approach to the problem, as follows. Denote the observed sample configuration at two loci by $\bfn$ and its sampling probability by $q(\bfn; \gr)$ (to be defined precisely below). Consider the asymptotic expansion in inverse powers of $\gr$:
\begin{equation}
\label{eq:main}
q(\bfn; \gr) = q_0(\bfn) + \frac{q_1(\bfn)}{\gr} + \frac{q_2(\bfn)}{\gr^2} + \cdots,
\end{equation}
where for convenience we suppress the dependence of these terms on other parameters of the model. Under an infinite-alleles type of mutation, we 
obtained closed-form formulas for $q_0(\bfn)$ and $q_1(\bfn)$ in terms of the marginal \emph{one}-locus sampling probabilities, and a decomposition of $q_2(\bfn)$ into a closed-form term plus a second part which is evaluated easily by dynamic programming \citep{jen:son:2010:AAP}. (The result is stated more precisely in \thmref{thm:jen:son:2009} below.) This provides the first closed-form extension of Ewens' Sampling Formula \citep{ewe:1972} to handle finite amounts of recombination. It has been extended subsequently to include more general models of mutation \citep{jen:son:2009:G}, natural selection \citep{jen:son:2012:AAP}, higher-order terms \citep{jen:son:2012:AAP}, and more than two loci \citep{bha:son:2012}, and has had practical implications for genomic inference \citep{cha:etal:2012}. One particularly appealing conclusion of these works is that both $q_0(\bfn)$ and $q_1(\bfn)$ are \emph{universal}; that is, their functional form is invariant to our assumptions about mutation and selection acting marginally at each locus. The effects of these marginal processes are entirely subsumed into the relevant one-locus sampling distributions. 

The simple and universal forms for $q_0(\bfn)$ and $q_1(\bfn)$ provide strong circumstantial evidence that there exists an underlying stochastic process which is much simpler than the standard models for finite amounts of recombination. In particular, we previously conjectured \citep{jen:son:2010:AAP} the existence of a process which is both much simpler than the standard models based on the Wright-Fisher diffusion or on the ARG, and is in agreement with the sampling distribution \eref{eq:main} up to $O(\gr^{-2})$.  The goal of this paper is to describe such a process. In fact, using different arguments we describe two such processes, obtaining both a limiting diffusion and a coalescent process with these properties. In the diffusion approximation, the key idea is to suppose that the probability $r$ of a recombination per individual per generation scales as $N^{-\gb}$ as the population size $N\to\infty$, for $0 < \gb < 1$, rather than the usual choice of $\gb = 1$. Interest in asymptotically large recombination rates is reasonable because of extensive recombination rate heterogeneity along chromosomes in e.g.\ humans, strong recombination rates in some species such as \emph{Drosophila melanogaster} \citep{cha:etal:2012}, and because of the need to understand the long-range dependencies between well-separated loci. Our diffusion in this scaling is intimately related to the central limit theorem for density dependent population processes \citep[see][Theorem 11.2.3]{eth:kur:1986}, which has been analyzed in genetics---for models of strong mutation rather than strong recombination---by \citet{fel:1951} and \citet{nor:1975:SIAM}. 
A closely related scaling in the context of $\Xi$-coalescent processes was also recently explored by \citet{bir:etal:2013} (in that paper $\gb = 1$ but with timescale $N^2$). The coalescent approach, meanwhile, uses a coupling argument. Intuitively, we would like to couple the ARG to the limiting case of two independent coalescent trees ($\gr = \infty$). To account for contributions to the sampling distribution of $O(\gr^{-1})$, we must quantify the ``leading order reasons'' for such a coupling to fail. When $\gr$ is large but finite, lineages in the ARG ancestral to both loci undergo recombination backwards in time very rapidly, until the first time $U$ that no such lineage survives. In this paper we show that, roughly speaking, in order to recover the sampling distribution up to $O(\gr^{-1})$ we need consider only the following type of exceptional event: \emph{a coalescence occurs more recently than time $U$ in the ARG, and the coalescence is between two lineages each of which is ancestral to both of the two loci}. This observation enables us to \emph{define} a simple coalescent process which allows for at most one of these events but is otherwise very similar to the easy limiting process corresponding to $\gr = \infty$.

The paper is organized as follows. In \sref{sec:notation} we specify our notation and summarize previous research. Novel diffusion and coalescent processes are introduced in Sections \ref{sec:diffusion} and \ref{sec:coalescent}, respectively, and we conclude in \sref{sec:discussion} with a brief discussion.

\section{Notation and previous results}
\label{sec:notation}
For $M \in \bbN = \{0,1,2,\ldots\}$, let $[M] := \{1,2,\ldots, M\}$. The complement of a set $J$ is written $J^\complement$. Denote the Kronecker delta by $\gd_{ij}$ which takes the value 1 if $i = j$ and 0 otherwise. Let $\bfe_i$ denote a unit vector whose $j$th entry is $\gd_{ij}$, and let $\bfe_{ij}$ denote a matrix with $(k,l)$th entry equal to $\gd_{ik}\gd_{jl}$. 
For a vector $\bfv \in \bbR^d$ we denote by $|\bfv|$ the usual Euclidean norm. 
Denote the $k \times l$ zero matrix by $\bfzero_{k\times l}$ and the $k \times k$ identity matrix by $\bfI_k$. We will replace a subscript with a ``$\cdot$'' to denote summation over that index. A prime symbol $'$ will denote vector or matrix transpose. For $z\in \bbR_{\geq 0}$ and $n \in \bbN$, $\phammer{z}{n} := z(z+1)\cdots(z+n-1)$ denotes the $n$th ascending factorial of $z$. Finally, for a matrix $\bfR$ of processes we let $[\bfR]_t = ([R_i,R_j]_t)_{i,j}$ denote the matrix of corresponding covariation processes. 

Consider the usual diffusion limit of an exchangeable model of random mating with constant population size of $N$ haplotypes. 
Our interest will be in a sample from this population at two loci, which we call A and B, with the probability of mutation per haplotype per generation denoted by $u_\A$ and $u_\B$ respectively. In the diffusion limit we let $N\to\infty$ and $u_\A$, $u_\B \to 0$ while the population-scaled parameters $\gq_\A = 2Nu_\A$ and $\gq_\B = 2Nu_\B$ remain fixed. In this paper we will suppose a \emph{finite-alleles} model of mutation such that a mutation to an allele $i$ in type space $E_\A = [K]$, $K \in \bbN$, takes it to allele $k \in [K]$ with probability $P_{ik}^\A$, with $E_\B = [L]$ and $P_{jl}^\B$, $j,l \in [L]$ defined analogously. (As we discover below, the mutation model is not important and we could pose something more complicated with little extra effort.) 
The probability of a recombination between the two loci per haplotype per generation is denoted by $r$, and we assume that $\gr_\gb = 2N^\gb r$ is fixed as $N \to \infty$, for some fixed $\gb \in (0,1]$. Previous work has focused on the case $\gb = 1$ with time measured in units of $N$ generations. For consistency with the usual notation we write $\gr = \gr_1$.

A sample from this model comprises $a$ haplotypes observed only at locus A, $b$ haplotypes observed only at locus B, and $c$ haplotypes observed at both loci. The sample configuration is denoted by $\bfn = (\bfa,\bfb,\bfc)$ where $\bfa = (a_i)_{i\in [K]}$ and $a_i$ is the number of haplotypes observed to exhibit allele $i$ at locus A; $\bfb = (b_j)_{j\in [L]}$ where $b_j$ is the number of haplotypes observed to exhibit allele $j$ at locus B; and $\bfc = (c_{ij})_{i\in [K], j\in [L]}$ where $c_{ij}$ is the number of haplotypes with allele $i$ at locus A and allele $j$ at locus B. Thus,
\[
\begin{array}{ccc}
a = \ds\si a_i, & b = \ds\sj b_j, & c = \ds\si\sj c_{ij},
\end{array}
\]
and we let $n = a + b + c$. We further write $\bfc_\A = (c_{i\cdot})_{i\in[K]}$ and $\bfc_\B = (c_{\cdot j})_{j\in[L]}$ to denote the marginal sample configurations of $\bfc$ restricted to locus A and locus B respectively. Finally, we use $q(\bfa,\bfb,\bfc)$ to denote the probability that when we sample $n$ haplotypes in some order from the population at stationarity we obtain the unordered configuration $(\bfa,\bfb,\bfc)$; by sampling exchangeability this is indeed a function only of the unordered configuration $(\bfa,\bfb,\bfc)$. For convenience we suppress the dependence of this quantity on the model parameters and on $\beta$. The main result motivating this work is an expansion for $q(\bfa,\bfb,\bfc)$ for the case of $\beta=1$, and later we will show that this expansion holds for all $\beta \in (0,1]$.

\begin{theorem}[See \citet{jen:son:2009:G}]
\label{thm:jen:son:2009}
Consider the following asymptotic expansion for $q(\bfa,\bfb,\bfc)$ under the diffusion limit with $\gb = 1$:
\[
q(\bfa,\bfb,\bfc) = q_0(\bfa,\bfb,\bfc) + \frac{q_1(\bfa,\bfb,\bfc)}{\gr} + O\left(\frac{1}{\gr^2}\right), \quad \text{as }\gr\to\infty,
\]
with $q_0$, $q_1$, $\ldots$ independent of $\gr$. Then the zeroth order term is given by
\begin{equation}
	q_0(\bfa,\bfb,\bfc) = q^\A(\bfa + \bfc_\A)q^\B(\bfb + \bfc_\B),
\label{eq:zerothorder}
\end{equation}
and the first order term is given by
	\begin{align}
		\label{eq:firstorder}
		q_1(\bfa,\bfb,\bfc) = {}& \binom{c}{2}q^\A(\bfa+\bfc_\A)q^\B(\bfb+\bfc_\B) \nonumber\\
	 	{}& -q^\B(\bfb+\bfc_\B)\sum_{i=1}^K \binom{c_{i\cdot}}{2}q^\A(\bfa+\bfc_\A-\bfe_i) \nonumber\\
	 	{}& -q^\A(\bfa+\bfc_\A)\sum_{j=1}^L \binom{c_{\cdot j}}{2}q^\B(\bfb+\bfc_\B-\bfe_j) \nonumber\\
	 	{}& +\sum_{i=1}^K \sum_{j=1}^L \binom{c_{ij}}{2}q^\A(\bfa+\bfc_\A-\bfe_i)q^\B(\bfb+\bfc_\B-\bfe_j),
	\end{align}
where $q^\A$, $q^\B$ are the marginal sampling distributions at locus A and locus B, respectively.
\end{theorem}

\begin{remark} \label{rem:jen:son:2009}
Under a neutral, finite-alleles model of mutation, if mutation is \emph{parent independent}---that is, $P_{ki}^\A = P_{i}^\A$, $i,k \in [K]$, and $P_{lj}^\B = P_{j}^\B$, $j,l \in [L]$, then $q^\A(\bfa)$ and $q^\B(\bfb)$ are known in closed-form:
\[
\begin{array}{lcr}
	q^\A(\bfa) = \ds\frac{1}{\phammer{\theta_\A}{a}}\prod_{i=1}^K \phammer{\theta_\A P^\A_i}{a_i}, & \text{ and } &
	q^\B(\bfb) = \ds\frac{1}{\phammer{\theta_\B}{b}}\prod_{j=1}^L \phammer{\theta_\B P^\B_j}{b_j}.
\end{array}
\]
These expressions follow, for example, from the moments of the Wright-Fisher diffusion with parent-independent mutation, whose stationary distribution at locus A is $\text{Dirichlet}(\tA P_1^\A,\ldots, \tA P_{K-1}^\A)$ \citep{wri:1949}, and similarly at locus B.
\end{remark}
\begin{remark}
The zeroth-order decomposition is well known \citep[e.g.][]{eth:1979:JAP} and also intuitive, since the two loci become independent as $\gr\to\infty$.
\end{remark}
\thmref{thm:jen:son:2009} can be obtained by diffusion \citep{jen:son:2012:AAP} or by coalescent \citep{jen:son:2009:G, jen:son:2010:AAP} arguments. In this paper we address both approaches in further detail.

\section{Diffusion model}
\label{sec:diffusion}
In this section we extend the above results by obtaining a full description of a simple diffusion process such that its sampling distribution is known \emph{exactly} and has a Taylor expansion about $\gr=\infty$ consistent with \eref{eq:zerothorder} and \eref{eq:firstorder}. For simplicity we will obtain our diffusion as the limit of an appropriately rescaled Moran model, although we expect our results to hold for a more general class of discrete models of reproduction within the domain of convergence of the Wright-Fisher diffusion. 

\subsection{Neutral Moran model}
A population of $N$ haploid, monoecious individuals evolves as a multitype birth-and-death process in continuous time. Each individual carries a haplotype comprising a pair of alleles $(i,j) \in [K] \times [L]$, one at locus A and one at locus B. Let $Z_{ij}(\ut) \in \{0,1,\ldots,N\}$ denote the number of $(i,j)$ haplotypes in the population at time $\ut\in \bbR_{\geq 0}$, and $\bfZ(\ut) = (Z_{ij}(\ut))_{i \in [K], j \in [L]}$. The population evolves as follows. At rate $N^2/2$ a reproduction event occurs, in which an individual is chosen uniformly at random from the population to die. It is replaced by a copy of another individual also chosen uniformly at random (the same individual could be chosen; whether sampling is with or without replacement does not affect the diffusion limit). Independently, each locus of each haplotype undergoes mutation: any locus A mutates at rate $\tA/2$ and its allele is updated according to the transition matrix $\bfP^\A = (P_{ik}^\A)_{i,k\in[K]}$; similarly any locus B mutates at rate $\tB/2$ and its allele is updated according to $\bfP^\B = (P_{jl}^\B)_{j,l\in[L]}$. Finally, each haplotype independently undergoes recombination at rate $\rho/2$: at such an event, it is replaced by a haplotype formed by sampling two alleles (one for each locus) independently from the population. Putting all this together, the rate at which a haplotype $(i,j)$ dies and is replaced by a haplotype $(k,l)$ when $\bfZ(\ut) = \bfz$ is given by
\begin{multline*}
\gl_{ij,kl}^{(N)}(\bfz) = \frac{z_{ij}}{N}\left[\frac{N^2}{2}\frac{z_{kl}}{N} + N\left(\frac{\tA}{2}P_{ik}^\A\gd_{jl} + \frac{\tB}{2}P_{jl}^\B\gd_{ik} + \frac{\gr}{2}\frac{z_{k\cdot}}{N}\frac{z_{\cdot l}}{N}\right)\right],\\ (i,j),(k,l) \in [K]\times [L].
\end{multline*}
Notice that, as is standard \citep[e.g.][]{baa:her:2008}, we decouple the mutation and recombination mechanisms from reproduction (and from each other). This simplifies the analysis without unduly affecting the diffusion limit.

We will change variables by introducing the collection
\[
\bfM^{(N)}(\ut) := \{\bfXN(\ut), \bfYN(\ut), \bfDN(\ut)\},
\]
where
\begin{align*}
\bfXN(\ut) = (X^{(N)}_i(\ut))_{i\in[K]} &= \left(\frac{Z_{i\cdot}(\ut)}{N}: i \in [K]\right),\\
\bfYN(\ut) = (Y^{(N)}_j(\ut))_{j\in[L]} &= \left(\frac{Z_{\cdot j}(\ut)}{N}: j \in [L]\right),\\
\bfDN(\ut) = (D^{(N)}_{ij}(\ut))_{i\in[K],j\in[L]} &= \left(\!\frac{Z_{ij}(\ut)}{N} - \frac{Z_{i\cdot}(\ut)}{N}\frac{Z_{\cdot j}(\ut)}{N}: i \in [K], j\in[L] \!\right).
\end{align*}
That is, we describe the state of the Moran model at time $\ut$ by the marginal allele frequencies 
and the coefficients of linkage disequilibrium \citep[see, e.g.][p69, p227]{ewe:2004:I}. 
We will write this succinctly by arranging the variables in a linear order:
\[
(X_1^{(N)},\ldots, X_K^{(N)}, Y_1^{(N)}, \ldots, Y_L^{(N)}, D^{(N)}_{11}, \ldots, D^{(N)}_{KL})',
\]
and thinking of $\bfM^{(N)}(\ut)$ as a vector of length $\gL := K + L + KL$. The process $(\bfM^{(N)}(\ut) : \ut = 0,1,\ldots)$ is then Markov on a state space we denote by $\gD_{KL-1}^{(N)}$, which is a rational subset (those points consistent with $\si\sj Z_{ij} = N$) of the $(KL - 1)$-dimensional shifted simplex
\begin{multline*}
\gD_{KL-1} = \Bigg\{ (\bfx,\bfy,\bfd) \in [0,1]^{K} \times [0,1]^L \times [-1,1]^{KL} : \\\si x_i = 1 = \sj y_j, \si d_{ij} = 0 = \sj d_{ij}\Bigg\}.
\end{multline*}
To find the diffusion limit we first need the conditional means and covariances of the increments 
\[
\gD\bfM^{(N)} (\ut) := \bfM^{(N)}(\ut + \dd\ut) - \bfM^{(N)}(\ut).
\]
From these, and under the assumption that $\tA$, $\tB$, and $\gr$ are fixed as $N\to\infty$, it is possible to show that the model converges to a (Wright-Fisher) diffusion limit \citep[Example 10.3.9, p433]{eth:kur:1986}. Recall however that our interest is when $\gr_\gb$, rather than $\gr$, is fixed, so below we write these increments in terms of $\gr_\gb$ using $\gr = \gr_\gb N^{1-\gb}$.

In the following, for convenience we drop the dependence on $\ut$.

\begin{proposition}
\label{prop:WF}
In the neutral two-locus Moran model with mutation and recombination, the conditional means and covariances of increments of $\bfM^{(N)}$ 
are given by
\allowdisplaybreaks
\begin{align}
\label{eq:Xi}\lim_{\dd\ut \to 0}(\dd\ut)^{-1}\bbE[\gD X_i^{(N)}\mid \bfM^{(N)}] = {}& \frac{\tA}{2}\sk (P_{ki}^\A - \gd_{ik})
X_k^{(N)}, \\
\lim_{\dd\ut \to 0}(\dd\ut)^{-1}\bbE[\gD Y_j^{(N)}\mid \bfM^{(N)}] = {}& \frac{\tB}{2}\sl (P_{lj}^\B - \gd_{jl})
Y_l^{(N)},\\
\lim_{\dd\ut \to 0}(\dd\ut)^{-1}\bbE[\gD D_{ij}^{(N)}\mid \bfM^{(N)}] = {} & -\frac{\gr_\gb}{2N^{\gb-1}}D^{(N)}_{ij} - D^{(N)}_{ij} \notag\\ {}&+ \frac{\tA}{2}\sk (P_{ki}^\A - \gd_{ik})
D^{(N)}_{kj} \notag\\
{}&  {}+ \frac{\tB}{2}\sl (P_{lj}^\B - \gd_{jl})
D^{(N)}_{il} + O\left(\frac{1}{N^{\gb}}\right),\label{eq:Dij}
\end{align}
\begin{align}
\lim_{\dd\ut \to 0}(\dd\ut)^{-1}\cov[\gD X_i^{(N)},\gD X_k^{(N)}\mid \bfM^{(N)}] = {}& X^{(N)}_i(\gd_{ik} - X^{(N)}_k) + O\left(\frac{1}{N^\gb}\right), \notag\\
\lim_{\dd\ut \to 0}(\dd\ut)^{-1}\cov[\gD Y_j^{(N)},\gD Y_l^{(N)}\mid \bfM^{(N)}] = {}& Y^{(N)}_j(\gd_{jl} - Y^{(N)}_l) + O\left(\frac{1}{N^\gb}\right),\notag\\
\lim_{\dd\ut \to 0}(\dd\ut)^{-1}\cov[\gD X_i^{(N)}, \gD Y_j^{(N)}\mid \bfM^{(N)}] = {}& D^{(N)}_{ij} + O\left(\frac{1}{N^\gb}\right),\notag\\
\lim_{\dd\ut \to 0}(\dd\ut)^{-1}\cov[\gD X_i^{(N)}, \gD D^{(N)}_{kl} \mid \bfM^{(N)}] = {}& D^{(N)}_{kl}(\gd_{ik} - X^{(N)}_i) - X^{(N)}_kD^{(N)}_{il} \notag\\ &{}+ O\left(\frac{1}{N^{\gb}}\right), \notag\\
\lim_{\dd\ut \to 0}(\dd\ut)^{-1}\cov[\gD Y_j^{(N)}, \gD D^{(N)}_{kl} \mid \bfM^{(N)}] = {}& D^{(N)}_{kl}(\gd_{jl} - Y^{(N)}_j) - Y^{(N)}_lD^{(N)}_{kj} \notag\\ &{}+ O\left(\frac{1}{N^{\gb}}\right),\notag\\
\lim_{\dd\ut \to 0}(\dd\ut)^{-1}\cov[\gD D_{ij}^{(N)}, \gD D_{kl}^{(N)} \mid \bfM^{(N)}] = {}& X_i^{(N)}Y_j^{(N)}(\gd_{ik} - X_k^{(N)})(\gd_{jl} - Y_l^{(N)}) \notag\\
& \hspace{-40pt} {}+ D^{(N)}_{kj}X_i^{(N)}Y_l^{(N)} + D_{il}^{(N)}X_k^{(N)}Y_j^{(N)} \notag\\
& \hspace{-40pt}{}+ D^{(N)}_{ij}(X_k^{(N)}Y_l^{(N)} - \gd_{ik}Y_l^{(N)} - \gd_{jl}X_k^{(N)})\notag\\
& \hspace{-40pt}{}+ D^{(N)}_{kl}(X_i^{(N)}Y^{(N)}_j - \gd_{ik}Y_j^{(N)} - \gd_{jl}X_i^{(N)}) \notag\\
& \hspace{-40pt}{}+ D_{ij}^{(N)}(\gd_{ik}\gd_{jl} - D^{(N)}_{kl}) 
+ O\left(\frac{1}{N^{\gb}}\right).\notag
\end{align}
Higher order moments of order $m\geq 2$ are $O(N^{-(m-2)})$.
\end{proposition}
\begin{proof}
These expressions follow directly from the first four moments of $\bfZ(\ut+\dd\ut)\mid \bfZ(\ut)$, which are easily computed by noting that
\begin{multline*}
\bbE[f(\bfZ(\ut + \dd\ut))\mid \bfZ(\ut) = \bfz] = \sum_{(i,j)}\sum_{(k,l)} f(\bfz - \bfe_{ij} + \bfe_{kl})\gl_{ij,kl}^{(N)}(\bfz) \dd\ut \\
{}+ f(\bfz)\left[1 - \frac{N}{2}(N + \tA + \tB + \gr)\dd\ut\right] + o(\dd\ut).
\end{multline*}
For example, choosing $f(\bfz) = z_{uv}$ we find
\begin{multline*}
\bbE[Z_{uv}(\ut + \dd\ut)\mid \bfZ(\ut) = \bfz] = z_{uv} \\ {}+ N\left[\frac{\tA}{2}\sk(P_{ku}^\A - \gd_{ku}) \frac{z_{kv}}{N} + \frac{\tB}{2}\sl(P_{lv}^\A - \gd_{lv}) \frac{z_{ul}}{N} + \frac{\gr}{2}\left(\frac{z_{u\cdot}}{N}\frac{z_{\cdot v}}{N} - \frac{z_{uv}}{N}\right)\right]\dd\ut \\{}+ o(\dd\ut),
\end{multline*}
and hence we recover \eqref{eq:Xi} via
\begin{align*}
\bbE[\gD X_u\mid \bfM^{(N)}] &= \frac{1}{N}\sum_{v=1}^L \left(\bbE[Z_{uv}(\ut + \dd\ut)\mid \bfZ(\ut)] - Z_{uv}\right) \\ &= \frac{\tA}{2}\sk(P_{ku}^\A - \gd_{ku}) X_{k}^{(N)}\dd\ut + o(\dd\ut).
\end{align*}
The remaining terms follow similarly; we omit the straightforward but lengthy algebraic details.
\end{proof}

To prepare for our diffusion limit, we must rescale time; from \eqref{eq:Dij} it is clear that to obtain a nontrivial limit we should let $t = N^{1-\gb}\ut$. Now introduce the conditional mean vector $\bfw^{(N)}$ and conditional covariance matrix $\bfs^{(N)}$ on this timescale, defined by
\begin{multline}
\lim_{\dd t \to 0}(\dd t)^{-1}\bbE[\gD\bfM^{(N)}\mid \bfM^{(N)}(t)=\bfm] = \\N^{\gb-1}\lim_{\dd\ut \to 0}(\dd\ut)^{-1}\bbE[\gD\bfM^{(N)}\mid \bfM^{(N)}(\ut)=\bfm]= :\bfw^{(N)}(\bfm),\label{eq:w}
\end{multline}
\begin{multline}
\lim_{\dd t \to 0}(\dd t)^{-1}\cov[\gD\bfM^{(N)}\mid \bfM(t)=\bfm] =\\ N^{\gb-1}\lim_{\dd\ut \to 0}(\dd\ut)^{-1}\cov[\gD\bfM^{(N)}\mid \bfM^{(N)}(\ut)=\bfm] =: N^{\gb-1}\bfs^{(N)}(\bfm), \label{eq:s}
\end{multline}
with entries determined by \propref{prop:WF}. Thus, with $\bfm = (x_1,\dots,x_K, \allowbreak y_1,\dots,y_L,  \allowbreak d_{11},\dots,d_{KL})$, equations \eref{eq:Xi}--\eref{eq:Dij} show that
\begin{align}
&&\bfw^{(N)}(\bfm) &= \bfw(\bfm) + O(N^{\gb-1}),\notag\\ \text{where }&&\bfw(\bfm) &= \Big(\underbrace{\vphantom{\frac{\gr_\gb}{2}}0,\ldots 0}_{K}, \underbrace{\vphantom{\frac{\gr_\gb}{2}}0, \ldots 0}_{L}, \underbrace{-\frac{\gr_\gb}{2}d_{11},\ldots, -\frac{\gr_\gb}{2}d_{KL}}_{K \times L}\Big)', \label{eq:w2}
\end{align}
with $\bfs^{(N)}(\bfm) = \bfs(\bfm) + O(N^{-\gb})$ determined in a similar fashion:
\[
\bfs(\bfm) = \begin{bmatrix}
\bfs_{\bf XX}(\bfm) & \bfs_{\bf XY}(\bfm) & \bfs_{\bf XD}(\bfm)\\
\bfs_{\bf XY}(\bfm) & \bfs_{\bf YY}(\bfm) & \bfs_{\bf YD}(\bfm)\\
\bfs_{\bf XD}(\bfm) & \bfs_{\bf YD}(\bfm) & \bfs_{\bf DD}(\bfm)
\end{bmatrix},
\]
where
\begin{align*}
[\bfs_{\bf XX}(\bfm)]_{ik} = {}& x_i(\gd_{ik} - x_k),\\
[\bfs_{\bf YY}(\bfm)]_{jl} = {}& y_j(\gd_{jl} - y_l),\\
[\bfs_{\bf XY}(\bfm)]_{ij} = {}& d_{ij},\\
[\bfs_{\bf XD}(\bfm)]_{i,kl} = {}& d_{kl}(\gd_{ik} - x_i) - x_kd_{il},\\
[\bfs_{\bf YD}(\bfm)]_{j,kl} = {}& d_{kl}(\gd_{jl} - y_j) - y_ld_{kj},\\
[\bfs_{\bf DD}(\bfm)]_{ij,kl} = {}& x_iy_j(\gd_{ik} - x_k)(\gd_{jl} - y_l) 
 {}+ d_{kj}x_iy_l + d_{il}x_ky_j  \\
& {}+ d_{ij}(x_ky_l - \gd_{ik}y_l - \gd_{jl}x_k) 
 {}+ d_{kl}(x_iy_j - \gd_{ik}y_j - \gd_{jl}x_i) \\
&{}+ d_{ij}(\gd_{ik}\gd_{jl} - d_{kl}).
\end{align*}
Notice in particular the different leading orders of the two quantities in \eqref{eq:w} and \eqref{eq:s}: 
the mean increments are of $O(1)$ on this timescale while the covariances are of $O(N^{\gb-1})$. It is this difference, which is a consequence of our assumption that the recombination probability $r$ is $O(N^{-\gb})$ for $\gb < 1$, that leads to a novel diffusion limit. Under the usual choice of $\gb = 1$ it is well known that we see convergence to a diffusion process after a linear rescaling of time. 
In the special case of a Wright-Fisher model and $K = L = 2$, the diffusion limit for $\bfM^{(N)}(\lfloor N \ut \rfloor)$ as $N\to\infty$ was obtained by \citet{oht:kim:1969:GRC, oht:kim:1969:G}. Our interest is however in $\gb \in (0,1)$, for which 
$r$ is larger, and the loss of linkage disequilibrium (LD) is subsequently much faster. Intuitively, we should expect such loss to resemble the exponential decay predicted in an infinitely large population, but with small fluctuations about this deterministic behaviour. The diffusion process we define below quantifies these fluctuations precisely.

\subsection{Gaussian diffusion limit of fluctuations in linkage disequilibrium}
We first provide a heuristic description of the diffusion limit. First, observe from \eqref{eq:w} and \eqref{eq:s} that, provided $\bfM^{(N)}(0) \to \bfM(0)$ as $N\to\infty$ and that $\gb \in [0,1)$, then
\begin{equation}
\label{eq:M}
\bfM^{(N)} \overset{d}{\to} \bfM := \left\{ (\bfX(0), \bfY(0), \bfD(0)e^{-\gr_\gb t/2})' : t\geq 0\right\}, \qquad N\to\infty,
\end{equation}
the deterministic exponential decay in LD typical of an infinitely large population. See \citet{baa:her:2008} for a formal statement of this law-of-large-numbers type result for the Moran model with recombination. For the corresponding central limit theorem, we seek a diffusion limit for
\begin{equation}
\label{eq:mainrescaling}
\bfU^{(N)}(t) := r_N[\bfM^{(N)}(t) - \bfM(t)],
\end{equation}
for some rescaling $r_N \to \infty$. In our application the appropriate choice is
\[
r_N := N^{(1-\gb)/2},
\]
which can be regarded as the one on which both recombination and genetic drift are observable on the fastest timescale \citep{jen:son:2012:AAP}. We will assume this scaling henceforward. To find the limit $\bfU = \lim_{N\to\infty} \bfU^{(N)}$, write
\begin{multline}
\label{eq:maindecomposition}
\bfU^{(N)}(t) = r_N\Bigg[[\bfM^{(N)}(0) - \bfM(0)] \\
{}+ \int_0^t [\bfw^{(N)}(\bfM^{(N)}(s)) - \bfw(\bfM(s))]\dd s + \Ma^{(N)}(t)\Bigg],
\end{multline}
where
\[
\Ma^{(N)}(t) := \bfM^{(N)}(t) - \bfM^{(N)}(0) - \int_0^t \bfw^{(N)}(\bfM^{(N)}(s)) \dd s
\]
describes the deviations of $\bfM^{(N)}(t)$ from its expected behaviour and is a martingale. It suffices to characterize the limits of each of the three grouped terms on the right of \eqref{eq:maindecomposition}. For the first term we assume that it converges to a limit, $\bfU^{(N)}(0) \overset{d}{\to} \bfU(0)$ as $N \to\infty$. For the second term, from \eqref{eq:w2} we should expect
\begin{align}
\MoveEqLeft r_N\int_0^t [\bfw^{(N)}(\bfM^{(N)}(s)) - \bfw(\bfM(s))]\dd s \notag\\
&= r_N\int_0^t\left[-\frac{\gr_\gb}{2}[\bfM^{(N)}(s) - \bfM(s)] +O(N^{\gb-1})\right] \dd s\notag\\
&= \int_0^t \left[-\frac{\gr_\gb}{2}\bfU^{(N)}(s) +O(N^{(\gb-1)/2})\right]\dd s \notag\\
&\overset{d}{\to} -\frac{\gr_\gb}{2}\int_0^t \bfU(s) \dd s, \qquad N\to\infty. \label{eq:middleterm}
\end{align}
Finally, we obtain a complete description of the limit $r_N\Ma^{(N)} \overset{d}{\to} \Ma$ as $N\to\infty$ by an application of the martingale central limit theorem \citep[Theorem 7.1.4]{eth:kur:1986}; we find
\[
\Ma(t) = \int_0^t \bfgs(\bfM(s)) \dd\bfW(s),
\]
where $\bfgs\bfgs' = \bfs$, and $\bfW$ is a $(KL-1)$-dimensional Brownian motion. In summary then, we expect $\bfU$ to satisfy
\begin{equation}
\label{eq:U}
\bfU(t) = \bfU(0) - \frac{\gr_\gb}{2}\int_0^t \bfU(s) \dd s + \int_0^t \bfgs(\bfM(s)) \dd\bfW(s).
\end{equation}
Our main result formalizes this argument, as follows. 

\begin{theorem}
\label{thm:nor:1975}
Suppose that $\bfU^{(N)}(0) \overset{d}{\to} \bfU(0)$ as $N\to\infty$. Then for each $t >0$, as $N\to\infty$,
\[
\sup_{s\leq t}|\bfM^{(N)}(s) - \bfM(s)| \overset{d}{\to} 0;
\]
$N^{(1-\gb)/2}\Ma^{(N)} \overset{d}{\to} \Ma$, where $\Ma$ has Gaussian, independent increments with mean zero, and with
\begin{equation}
\label{eq:martingale-covariance}
\bbE[\Ma(t)\Ma(t)'] = \int_0^t \bfs(\bfM(s)) \dd s;
\end{equation}
and $\bfU^{(N)} \overset{d}{\to} \bfU$, satisfying \eqref{eq:U}.
\end{theorem}

\begin{proof}[Proof of \thmref{thm:nor:1975}]
This is an application of a central limit theorem for density dependent population processes; for textbook coverage see \citet[Chapter 11]{eth:kur:1986} and for a recent treatment see \citet{kan:etal:2014}. We apply Theorem 2.11 of \citet{kan:etal:2014}. To do so we need to validate each of the assertions that led to \eqref{eq:U} above by checking the following sufficient conditions (i)--(iv). 
(\citet[][Theorem 2.11]{kan:etal:2014} is rather more general than is required here: it permits the state space of $\bfM^{(N)}$ to be unbounded, and for $\bfM^{(N)}$ to depend on other processes that evolve on faster timescales than that of the diffusion. We omit those conditions which are not needed.)

\noindent \emph{(i) {\bf The Moran process converges to an identifiable, deterministic limit.} This is guaranteed by the following: 
the infinitesimal generator $\cA_N$ of $\bfM^{(N)}$ satisfies
\[
\lim_{N\to\infty} \sup_{\bfm \in 
\gD_{KL-1}^{(N)}}\left|\cA_N f(\bfm) - \cA f(\bfm)\right| = 0, \qquad f \in \cD(\cA),
\]
for a generator $\cA$ with domain $\cD(\cA)$.}

\noindent \emph{(ii) {\bf Fluctuations about the deterministic limit are well behaved.} More precisely, $\bfR^{(N)}$ is a local martingale and the covariations processes $[\bfM^{(N)}]_t \overset{d}{\to} 0$.}

\noindent \emph{(iii) {\bf Contributions of $O(r_N^{-1})$ to the error $\bfw^{(N)} - \bfw$ can be identified.} These would contribute to the limiting drift of $\bfU(t)$, and a sufficient condition to identify them is: there exists a continuous function $\bfG_0:\gD_{KL-1} \to \bbR^{\gL}$ (recall $\gL = K+L+KL$) such that 
\[
\lim_{N\to\infty} \sup_{\bfm\in 
\gD_{KL-1}^{(N)}}\left|r_N\left[\bfw^{(N)}(\bfm) - \bfw(\bfm)\right] - \bfG_0(\bfm)\right| = 0.
\]}
\noindent \emph{(iv) {\bf The martingale central limit theorem applies to $r_N \bfR^{(N)}$.} This is guaranteed by the following:
\begin{equation}
\label{eq:jumpsize}
\lim_{N\to\infty}\bbE\left[\sup_{s\leq t} r_N \left| \bfM^{(N)}(s) - \bfM^{(N)}(s-) \right|\right] = 0,
\end{equation}
and there exists a continuous $\bfG:\gD_{KL-1} \to \bbR^{\gL\times \gL}$ such that for each $t > 0$,
\begin{equation}
\label{eq:covariation}
r_N^2[\bfM^{(N)}]_t - \int_0^t \bfG(\bfM^{(N)}(s)) \dd s \overset{d}{\to} 0.
\end{equation}
}We address each of these requirements in turn.

\emph{(i)} Convergence of $\cA_Nf(\bfm)$ to $\cA f(\bfm) := \bfw.\nabla f(\bfm)$, the generator of $\bfM$ [see \eqref{eq:M}], is immediate from \propref{prop:WF}. Convergence is uniform in $\bfm$ because the $O(N^{-\gb})$ terms in \propref{prop:WF} have coefficients that are polynomials in $\bfM^{(N)}$ on a compact space. 

\emph{(ii)} Since the state space is bounded, for $\bfR^{(N)}$ to be a martingale it suffices that the jump rate is uniformly bounded \citep[Proposition 2.1]{kur:1971}, as is the case for the Moran process. The covariations process $[\bfM^{(N)}]_t \overset{d}{\to} 0$ as a consequence of \eqref{eq:covariation}, verified below.

\emph{(iii)} From \eqref{eq:w2}, $r_N[\bfw^{(N)}(\bfm) - \bfw(\bfm)] = O(N^{(\gb - 1)/2})$, again uniformly in $\bfm \in \gD_{KL-1}^{(N)}$, so here the appropriate choice is $\bfG_0 \equiv \bfzero$. Thus, the only relevant contribution to the limit \eqref{eq:middleterm} is from the error $\bfw(\bfM^{(N)}(s)) - \bfw(\bfM(s))$ rather than from $\bfw^{(N)}(\bfM^{(N)}(s)) - \bfw(\bfM^{(N)}(s))$.

\emph{(iv)} Jumps of any component of $\bfM^{(N)}$ are bounded in magnitude by $2/N$, so
\[
\sup_{s\leq t} r_N \left| \bfM^{(N)}(s) - \bfM^{(N)}(s-) \right| \leq N^{(1-\gb)/2}\cdot \frac{2\gL^{1/2}}{N} \to 0, \qquad N\to\infty,
\]
and \eqref{eq:jumpsize} holds. To identify the asymptotic behaviour of $r_N^2[\bfM^{(N)}]_t$, let
\[
\cN^{(N)}_\bfm(t) = \cY_\bfm\left(\int_0^t \gl^{(N)}_\bfm(\bfM^{(N)}(s)) \dd s\right)
\]
denote the total number of jumps of the Moran process into state $\bfm \in \gD_{KL-1}^{(N)}$ by time $t$, where $(\cY_\bfm : \bfm \in \gD_{KL-1}^{(N)})$ is a collection of independent Poisson processes of unit rate and $\gl_\bfm^{(N)}(\bfM^{(N)}(s))$ denotes the rate of transition of the process from current state $\bfM^{(N)}(s)$ to $\bfm$. Then
\begin{align*}
\MoveEqLeft r_N^2[\bfM^{(N)}]_t = N^{1-\gb}\int_0^t\sum_{\bfm \in \gD_{KL-1}^{(N)}} [\gD\bfM^{(N)}(s)][\gD\bfM^{(N)}(s)]' \dd\cN^{(N)}_\bfm(s),\\
&\sim N^{1-\gb}\int_0^t\sum_{\bfm \in \gD_{KL-1}^{(N)}} [\bfm-\bfM^{(N)}(s)][\bfm -\bfM^{(N)}(s)]' \gl_\bfm^{(N)}(\bfM^{(N)}(s)) \dd s,\\
&\sim \int_0^t \bfs^{(N)}(\bfM^{(N)}(s)) \dd s, 
\end{align*}
by \eqref{eq:s}. Thus we may take $\bfG = \bfs$ in \eqref{eq:covariation} [$\bfG$ identifies the moments appearing in \eqref{eq:martingale-covariance}].
\end{proof}

\begin{remark}
One could obtain the same diffusion limit starting from a Wright-Fisher model rather than a Moran model, since the means and covariances of its increments are identical to leading order, up to a rescaling of time. This alternative approach is in some respects less appealing since the Wright-Fisher model, when expressed in continuous time, is non-Markovian. The additional complications raised by this approach have been addressed by \citet{nor:1975:SIAM} \citep[see also][]{eth:nag:1980, eth:nag:1988}, and we have checked that the conditions of his theorems still apply when we introduce recombination to the Wright-Fisher model. The theory of \citet{nor:1975:SIAM} has been used to study strong mutation and selection \citep{nor:1972, nor:1975:SIAM, kap:etal:1988, nag:1986, nag:1990, wak:sar:2009}, and a Gaussian diffusion approximation of a Moran model with strong selection is developed by \citet{fed:etal:2014}, but to the best of our knowledge this is the first time a central limit theorem has been obtained for strong recombination.
\end{remark}
\begin{remark}
The exponential decay of linkage disequilibrium implied by $\bfM$ [equation \eqref{eq:M}] is a classical result; the above theorem further quantifies the fluctuations about this deterministic behaviour in a fully time-dependent manner. In particular, the definition of $\bfU$ [equation \eqref{eq:mainrescaling}] shows that fluctuations are of order $N^{(1-\gb)/2}$ on a timescale of $N^{\gb-1}$ units of the Moran process. If we designate the expected lifetime of an individual, $2/N$, as one \emph{generation}, then these fluctuations can be said to occur on a timescale of order $N^\gb$ generations. (This definition of ``generation'' is consistent with $u_\A$, $u_\B$, and $r$ in \sref{sec:notation} provided we replace $N$ with the effective population size of the Moran model, $N/2$, in the definitions of $\tA$, $\tB$, and $\gr$ \citep[p121]{ewe:2004:I}.)
\end{remark}
\subsection{Stationary distribution}
Although $\bfU$ is described completely by \eqref{eq:U}, the volatility term $\bfgs(\bfM(t))$ is neither simple nor time-independent. On the other hand, our main interest is in stationary behaviour, and $\bfgs(\bfM(\infty))$ takes on a much simpler form. First note that the components of $\bfU(t)$ corresponding to each $X_i$ and $Y_j$ undergo Brownian motions (with nonunit volatility), so we restrict our attention to the stationary distribution of the component corresponding to $\bfD$, which we denote $\bfU_{\bfD}$. Conditions of \citet{nor:1975:AAP} 
confirm convergence of $\bfU_{\bfD}(t)$ to its stationary distribution. Setting $\bfgs(\bfM(s)) = \bfgs(\bfM(\infty))$ in \eqref{eq:U}, we find
\begin{equation}
\label{eq:OU}
\dd\bfU_{\bfD} = -\frac{\gr_\gb}{2}\bfU_{\bfD}\dd t + \bfgs_{\infty} \dd\bfW(s),
\end{equation}
where $\bfgs_{\infty}$ is a \emph{constant} defined by
\[
\bfgs_{\infty}\bfgs_{\infty}' = \bfs_\infty := \bfs_{\bfD\bfD}(\bfM(\infty)) = [X_i(0)Y_j(0)(\gd_{ik}-X_k(0))(\gd_{jl}-Y_l(0))]_{ij,kl}.
\]
The process \eqref{eq:OU} is much simpler to describe. Marginally, $\bfU_{D_{ij}}$ is an Ornstein-Uhlen\-beck process with damping towards linkage equilibrium at rate $\gr_\gb/2$ and constant volatility $[\bfgs_\infty]_{ij,ij}$. 
$\bfU_{\bfD}$ has stationary distribution
\[
\text{Normal}\left(\bfzero_{KL\times 1}, \frac{\bfs_\infty}{\gr_\gb}\right).
\]
This is a slightly different idea of stationarity than usual, since it depends on $\bfX(0)$ and $\bfY(0)$. An immediate question is: what should be the distributions for $\bfX(0)$ and $\bfY(0)$? We address this by reconsidering the usual two-locus \emph{Wright-Fisher} diffusion limit operating on a slower timescale. We can exploit \eqref{eq:OU} to obtain a simple approximation of this diffusion limit, as follows. First, we have \emph{derived} the Gaussian diffusion approximation
\[
\bfD(0)e^{-\gr_\gb t/2} + N^{(\gb-1)/2}\bfU_{\bfD}(t)
\]
for $\bfD^{(N)}(t)$. 
Thus the stationary distribution of this approximation is
\begin{equation}
\label{eq:OUstationarity}
\text{Normal}\left(\bfzero_{KL\times 1}, \frac{\bfs_\infty}{\gr}\right).
\end{equation}
Notice that this description does not depend on the particular choice of $\gb$. Under the usual ``Wright-Fisher'' regime we treat $\gr$ as fixed. 
It remains to specify the stationary distributions for the marginal allele frequencies $\bfX$ and $\bfY$, which we suppose to have reached their usual (independent) stationary distributions in the Wright-Fisher diffusion limit, which we refer to as $\pi_{\A}$ and $\pi_{\B}$, respectively (and whose respective sampling distributions are $q^\A$ and $q^\B$). Then we can complete the picture for \eref{eq:OUstationarity} by specifying $(\bfX(0),\bfY(0)) \sim \pi_{\A}\otimes \pi_{\B}$.

The distribution \eref{eq:OUstationarity} therefore provides a simple, explicit method for the approximate simulation of haplotype frequencies under a stationary, two-locus Wright-Fisher diffusion, which we summarize in the algorithm below. (When mutation is parent independent, as in Remark \ref{rem:jen:son:2009}, $\pi_{\A}$ and $\pi_{\B}$ take on a particularly simple form, but we note that these distributions are not known in general.)
\\

\fbox{
\begin{minipage}[c]{0.9\textwidth}
{\bf Algorithm to simulate from a Gaussian approximation to the stationary Wright-Fisher diffusion with recombination.}
\begin{enumerate}
\item Simulate marginal allele frequencies at locus A, $\bfX(0) \sim \pi_{\A}$.
\item Independently simulate marginal allele frequencies at locus B,\\ $\bfY(0) \sim \pi_{\B}$.
\item Conditionally simulate $\bfD$ from \eref{eq:OUstationarity} given $\bfX(0)$ and $\bfY(0)$.
\item Calculate two-locus haplotype frequencies via
\[
X_{ij} = D_{ij} + X_i(0)Y_j(0), \quad \text{for each }i\in[K], j\in[L].
\]
\end{enumerate}
\end{minipage}}
\\

\subsection{Sampling distribution} The significance of the Gaussian diffusion approximation $\bfU_{\bfD}$ is further evident from the following theorem. First we need some further notation. Let
\[
\partition{m} = \left\{\bfr\in \bbN^{K\times L}: \sum_{i=1}^K\sum_{j=1}^L r_{ij} = m\right\},
\]
for $m \in \bbN$, and let $\haplist{\bfr} \in ([K]\times [L])^m$ denote a sequence of $m$ haplotypes (in some arbitrary, fixed order) with multiplicities specified by $\bfr \in \partition{m}$. Further let $\haplistA{\bfr} \in [K]^m$ denote the corresponding list of alleles 
obtained by looking at the first entry of each element of $\haplist{\bfr}$, and define $\haplistB{\bfr}$ similarly. For $\gl \in \bbN$ denote by $\partitiontwo{2\gl}$ the set of partitions of $[2\gl]$ with precisely $\gl$ blocks of size $2$, and write a representative element as $\bfxi_{\bfmu\bfnu} = \{\{\mu_k,\nu_k\}: k=1,\ldots,\gl\}\in \partitiontwo{2\gl}$; $\bfmu = (\mu_k)$ and $\bfnu = (\nu_k)$ are sequences of length $\gl$. 
For $J \subseteq [\gl]$, denote by $\bfmu_J$, $\bfnu_J$ the subsequences obtained by looking only at the indices in $J$, and denote by $\haplist{\bfr}_{\bfmu}$ the subsequence of $\haplist{\bfr}$ obtained by looking only at the indices in $\bfmu$. The matrix of multiplicities of $\haplist{\bfr}_{\bfmu}$ is denoted by $\bfr^{(\bfmu)}$, so that $\bfr^{(\bfmu)} + \bfr^{(\bfnu)} = \bfr$. For example, if $\bfr = [\begin{smallmatrix} 1 & 2\\ 0 & 1 \end{smallmatrix}]$ then a representative list of haplotypes is $\haplist{\bfr} = ((1,1),(1,2),(1,2),(2,2))$ with marginal allele lists $\haplistA{\bfr} = (1,1,1,2)$ and $\haplistB{\bfr} = (1,2,2,2)$. Here, $m = 2\gl = 4$, and $\partitiontwo{4} = \big\{\{\{1,2\},\{3,4\}\}, \{\{1,3\},\{2,4\}\}, \{\{1,4\},\{2,3\}\}\big\}$. Then for example the first element in $\partitiontwo{4}$ is the partition $\bfxi_{\bfmu\bfnu}$ constructed from $\bfmu = (1,3)$ and $\bfnu = (2,4)$, and so $\haplist{\bfr}_{\bfmu} = ((1,1),(1,2))$ and $\haplist{\bfr}_{\bfnu} = ((1,2),(2,2))$.

\begin{theorem}
\label{thm:gaussiandiffusion}
Suppose that $\bfX \sim \pi_{\A}$, $\bfY \sim \pi_{\B}$ independently, and conditional on $\bfX$ and $\bfY$, $\bfD$ is distributed according to the Gaussian distribution in \eqref{eq:OUstationarity}. Then the sampling distribution is given \emph{exactly} by
{\allowdisplaybreaks
\begin{align}
q_{G}(\bfa,\bfb,\bfc) = {}& \sum_{\gl=0}^{\lfloor c/2 \rfloor} \frac{1}{\gr^{\gl}}\sum_{\bfr\in \partition{2\gl}} \sum_{\bfxi\in \partitiontwo{2\gl}}\left[\prod_{i=1}^K \prod_{j=1}^L\binom{c_{ij}}{r_{ij}}\right]\notag\\
{}& \times \Bigg[\sum_{I \subseteq [\gl]:\,\haplistA{\bfr}_{\bfmu_I} = \haplistA{\bfr}_{\bfnu_I}} (-1)^{|I^\complement|}q^\A(\bfa + \bfc_\A-\bfr_\A^{(\bfnu_I)})\Bigg]\notag\\
& {}\times \Bigg[\sum_{J\subseteq [\gl]:\,\haplistB{\bfr}_{\bfmu_J}=\haplistB{\bfr}_{\bfnu_J}}(-1)^{|J^\complement|}q^\B(\bfb + \bfc_\B - \bfr_\B^{(\bfnu_J)})\Bigg], \label{eq:diffusionexact}\\
= {}& q_0(\bfa,\bfb,\bfc) + \frac{q_1(\bfa,\bfb,\bfc)}{\gr} + O\left(\frac{1}{\gr^2}\right), \notag
\end{align}}
with $q_0$ and $q_1$ given by \eref{eq:zerothorder} and \eref{eq:firstorder} respectively (and we impose the convention that the empty summations for $\gl = 0$ have a single term, with $(-1)^{|\varnothing\setminus\varnothing|} = 1$). 
\end{theorem}
\begin{proof}
With respect to the diffusion in the transformed co-ordinate system, the sampling distribution is
{\allowdisplaybreaks
\begin{align*}
q_{G}(\bfa,\bfb,\bfc) 
={} & \bbE\left[\Bigg(\prod_{i=1}^K X_{i}^{a_i}\Bigg)\Bigg(\prod_{j=1}^L Y_{j}^{b_j}\Bigg)\Bigg(\prod_{i=1}^K \prod_{j=1}^L \left[D_{ij} + X_{i}Y_{j}\right]^{c_{ij}}\Bigg)\right], \\
	={} & \sum_{m=0}^c \sum_{\bfr\in \partition{m}} \left[\prod_{i=1}^K \prod_{j=1}^L\binom{c_{ij}}{r_{ij}}\right]
	\bbE\left[\Bigg(\prod_{i=1}^K X_{i}^{a_i+c_{i\cdot}-r_{i\cdot}}\Bigg)\right.\\
	& \specialcell{\hfill {}\times \left. \Bigg(\prod_{j=1}^L Y_{j}^{b_j+c_{\cdot j} - r_{\cdot j}}\Bigg) \bbE\left[\prod_{i=1}^K \prod_{j=1}^L D_{ij}^{r_{ij}}\mid \bfX, \bfY\right]\right],}\\
	={} & \sum_{\gl=0}^{\lfloor c/2\rfloor} \sum_{\bfr\in \partition{2\gl}} \sum_{\bfxi\in \partitiontwo{2\gl}}\left[\prod_{i=1}^K \prod_{j=1}^L\binom{c_{ij}}{r_{ij}}\right]\bbE\left[\Bigg(\prod_{i=1}^K X_{i}^{a_i+c_{i\cdot}-r_{i\cdot}}\Bigg)\right.\\
	& \specialcell{\hfill {}\times \left. \Bigg(\prod_{j=1}^L Y_{j}^{b_j+c_{\cdot j} - r_{\cdot j}}\Bigg) \prod_{k=1}^{\gl} \bbE[D_{\haplist{\bfr}_{\mu_k}}D_{\haplist{\bfr}_{\nu_k}}\mid \bfX, \bfY]\right],}\\
	={} & \sum_{\gl=0}^{\lfloor c/2 \rfloor} \frac{1}{\gr^{\gl}}\sum_{\bfr\in \partition{2\gl}} \sum_{\bfxi\in \partitiontwo{2\gl}}\left[\prod_{i=1}^K \prod_{j=1}^L\binom{c_{ij}}{r_{ij}}\right] \\
	& \times \bbE\left[\Bigg(\prod_{i=1}^K X_{i}^{a_i+c_{i\cdot}-r_{i\cdot}}\Bigg)\Bigg(\prod_{j=1}^L Y_{j}^{b_j+c_{\cdot j} - r_{\cdot j}}\Bigg)\right.\\
	& \specialcell{\hfill \left. {}\times \prod_{k=1}^{\gl}X_{\haplistA{\bfr}_{\mu_k}}Y_{\haplistB{\bfr}_{\mu_k}}(\gd_{\haplistA{\bfr}_{\mu_k} \haplistA{\bfr}_{\nu_k}} - X_{\haplistA{\bfr}_{\nu_k}})(\gd_{\haplistB{\bfr}_{\mu_k}\haplistB{\bfr}_{\nu_k}} - Y_{\haplistB{\bfr}_{\nu_k}})\right],}\\
	={}& \sum_{\gl=0}^{\lfloor c/2 \rfloor} \frac{1}{\gr^{\gl}}\sum_{\bfr\in \partition{2\gl}} \sum_{\bfxi\in \partitiontwo{2\gl}}\left[\prod_{i=1}^K \prod_{j=1}^L\binom{c_{ij}}{r_{ij}}\right] \\
	& {}\times \sum_{I \subseteq [\gl]} (-1)^{|I^\complement|}\gd_{\haplistA{\bfr}_{\bfmu_I}\haplistA{\bfr}_{\bfnu_I}}\sum_{J\subseteq [\gl]}(-1)^{|J^\complement|}\gd_{\haplistB{\bfr}_{\bfmu_J}\haplistB{\bfr}_{\bfnu_J}}\\
	& {} \times \bbE\left[\Bigg(\prod_{i=1}^K X_{i}^{a_i+c_{i\cdot}-r^{(\bfnu_I)}_{i\cdot}}\Bigg)\Bigg(\prod_{j=1}^L Y_{j}^{b_j+c_{\cdot j} - r^{(\bfnu_J)}_{\cdot j}}\Bigg) \right],
\end{align*}}
The second equality follows from the multinomial theorem and the tower property, the third equality follows from Isserlis' theorem \citep{mic:etal:2011}, 
and the fourth equality follows from \eref{eq:OUstationarity}:
\[
\bbE[D_{ij}D_{kl}\mid \bfX, \bfY] = \frac{1}{\gr}X_iY_j(\gd_{ik} - X_k)(\gd_{jl} - Y_l).
\]
The fifth equality follows from expanding the final product (using the convention $\gd_{\varnothing\varnothing} = 1$), while \eref{eq:diffusionexact} follows from $(\bfX,\bfY) \sim \pi_\A \otimes \pi_\B$. The equalities still hold for $\gl = 0$ provided we take $\prod_\varnothing = 1$.

Extracting the two leading order terms $\gl=0$ and $\gl = 1$, the expression simplifies to
\begin{align*}
q_{G}(\bfa,\bfb,\bfc) = {}& \bbE\left[\Bigg(\prod_{i=1}^K X_{i}^{a_i+c_{i\cdot}}\Bigg)\Bigg(\prod_{j=1}^L Y_{j}^{b_j+c_{\cdot j}}\Bigg)\right]\\
& {}+ \frac{1}{\gr}\sum_{k,u=1}^K\sum_{l,v=1}^L \frac{c_{kl}(c_{uv} - \gd_{ku}\gd_{lv})}{2} \bbE\left[\Bigg(\prod_{i=1}^K X_{i}^{a_i+c_{i\cdot}-\gd_{iu}}\Bigg)\right.\\
& \specialcell{\hfill \left. {}\times \Bigg(\prod_{j=1}^L Y_{j}^{b_j+c_{\cdot j} - \gd_{jv}}\Bigg) (\gd_{ku} - X_{u})(\gd_{lv} - Y_{v})\right] + O\left(\frac{1}{\gr^2}\right),}\\
= {}& q_0(\bfa,\bfb,\bfc) + \frac{q_1(\bfa,\bfb,\bfc)}{\gr} + O\left(\frac{1}{\gr^2}\right),
\end{align*}
as required.
\end{proof}

\subsection{Accuracy of the diffusion process}
A natural question to ask is: to what extent does the process of \thmref{thm:gaussiandiffusion} capture the dynamics of the full process? To address this we consider the accuracy of the sampling distribution \eref{eq:diffusionexact} as an approximation to the ``true'' distribution, $q(\bfa,\bfb,\bfc)$. For moderate sample sizes it is possible to compute the latter as the solution to a system of recursive equations \citep{gol:1984, eth:gri:1990, jen:son:2009:G}. The number of summands in \eref{eq:diffusionexact} grows rapidly with $\gl$ (as long as $\gl \leq \lfloor \frac{c}{2}\rfloor$), so we define an approximate sampling distribution $q_{\G}^{(\gl)}(\bfa,\bfb,\bfc)$ by truncating the outer sum in \eref{eq:diffusionexact} at a fixed index $\gl$. This is analogous to the asymptotic sampling formulae for the full model which are obtained by truncating equation \eref{eq:main} \citep{jen:son:2012:AAP}. As our measure of accuracy we define the relative error,
\begin{equation}
\label{eq:relativeerror}
\eG{\gl} = \left| \frac{Q_{\G}^{(\gl)}(\bfzero,\bfzero,\bfc) - q(\bfzero,\bfzero,\bfc)}{q(\bfzero,\bfzero,\bfc)}\right| \times 100\%,
\end{equation}
where $Q_{\G}^{(\gl)}(\bfzero,\bfzero,\bfc)$ is the staircase Pad\'e approximant to $q_{\G}^{(\gl)}(\bfzero,\bfzero,\bfc)$. \citep[The former is used for its superior convergence properties; see][for details.]{jen:son:2012:AAP} We define $\eT{\gl}$ analogously, replacing $Q_{\G}^{(\gl)}(\bfzero,\bfzero,\bfc)$ in \eref{eq:relativeerror} with the Pad\'e approximant to the partial sum of \eref{eq:main}, computed up to $O(\gr^{-(\gl+1)})$ by the method of \citet{jen:son:2012:AAP}.

We computed the distribution of $\eG{\gl}$ and of $\eT{\gl}$ across all sample configurations of size $c=20$ for which both alleles are observed at each locus; results are shown in \tref{tab:accuracy}. For a collection of this size it was straightforward to compute up to $\gl = 6$ for every possible sample configuration. Using a partial sum to approximate \eref{eq:main} contributes to both errors; $\eG{\gl}$ has additional contributions reflecting its use of an approximate \emph{model}. Of course, the two errors agree up to $\gl =1$. However, \tref{tab:accuracy} shows that they are comparable more broadly, particularly for large recombination rates. As $\gl$ increases, $Q_{\G}^{(\gl)}(\bfzero,\bfzero,\bfc)$ converges rapidly 
(even without Pad\'e summation; not shown), and becomes a reasonable approximation to $q(\bfa,\bfb,\bfc)$. For example, for $\gr = 50$, $Q^{(6)}_{\G}(\bfzero,\bfzero,\bfc)$ is within $10\%$ of $q(\bfa,\bfb,\bfc)$ with probability $0.79$, though it is within $1\%$ only with probability $0.50$. When we consider the highest levels of accuracy, as in $\Phi(1)$ in \tref{tab:accuracy}, $\eG{\gl}$ actually increases with $\gl$ when $\gl > 1$. This suggests that the Gaussian model typically cannot approximate the true model to the same level of precision as a first order asymptotic approximation of the true model, though its behaviour as a coarser approximation (as reflected in the columns for  $\Phi(100)$, for example) is comparable.


\begin{table}[t]
\begin{center}
\caption{\label{tab:accuracy}Cumulative distribution $\Phi(x) = \bbP(\eApprox^{(\gl)} < x\%)$ (where $\eApprox^{(\gl)}$ denotes either $\eG{\gl}$ or $\eT{\gl}$ as defined in the main text), for all samples of size $20$ dimorphic at both loci.} 
\label{tab:main}
\begin{tabular}{ccccc|cccc}
& & \multicolumn{3}{c}{$\gr = 25$} & \multicolumn{3}{c}{$\gr = 50$}\\
\hline
	& Type&&&\\
$\gl$	& of sum & $\Phi(1)$ & $\Phi(10)$ & $\Phi(100)$& $\Phi(1)$ & $\Phi(10)$ & $\Phi(100)$ \\\hline
0	&	True	&	0.39	&	0.58	&	1.00	&	0.49	&	0.63	&	1.00	\\
	&	Gaussian	&	0.39	&	0.58	&	1.00	&	0.49	&	0.63	&	1.00	\\
1	&	True	&	0.51	&	0.75	&	0.96	&	0.59	&	0.84	&	0.99	\\
	&	Gaussian	&	0.51	&	0.75	&	0.96	&	0.59	&	0.84	&	0.99	\\
2	&	True	&	0.59	&	0.91	&	0.97	&	0.77	&	0.98	&	1.00	\\
	&	Gaussian	&	0.50	&	0.73	&	0.97	&	0.50	&	0.86	&	1.00	\\
4	&	True	&	0.83	&	0.99	&	1.00	&	0.95	&	1.00	&	1.00	\\
	&	Gaussian	&	0.51	&	0.72	&	1.00	&	0.50	&	0.80	&	1.00	\\
6	&	True	&	0.89	&	0.99	&	1.00	&	0.99	&	1.00	&	1.00	\\
	&	Gaussian	&	0.49	&	0.71	&	0.99	&	0.50	&	0.79	&	1.00	\\
\hline\\[-5pt]
& & \multicolumn{3}{c}{$\gr = 100$} & \multicolumn{3}{c}{$\gr = 200$}\\
\hline
	& Type&&&\\
$\gl$	& of sum & $\Phi(1)$ & $\Phi(10)$ & $\Phi(100)$& $\Phi(1)$ & $\Phi(10)$ & $\Phi(100)$ \\\hline
0	&	True	&	0.50	&	0.72	&	1.00	&	0.54	&	0.95	&	1.00	\\
	&	Gaussian	&	0.50	&	0.72	&	1.00	&	0.54	&	0.95	&	1.00	\\
1	&	True	&	0.74	&	0.95	&	1.00	&	0.90	&	0.99	&	1.00	\\
	&	Gaussian	&	0.74	&	0.95	&	1.00	&	0.90	&	0.99	&	1.00	\\
2	&	True	&	0.95	&	1.00	&	1.00	&	1.00	&	1.00	&	1.00	\\
	&	Gaussian	&	0.64	&	0.99	&	1.00	&	0.85	&	1.00	&	1.00	\\
4	&	True	&	1.00	&	1.00	&	1.00	&	1.00	&	1.00	&	1.00	\\
	&	Gaussian	&	0.64	&	0.99	&	1.00	&	0.83	&	1.00	&	1.00	\\
6	&	True	&	1.00	&	1.00	&	1.00	&	1.00	&	1.00	&	1.00	\\
	&	Gaussian	&	0.64	&	0.99	&	1.00	&	0.83	&	1.00	&	1.00	\\
\hline
\end{tabular}
\end{center}
\end{table}

\section{Coalescent process}
\label{sec:coalescent}
\subsection{A coupling argument}
In this section we derive a coalescent process which is much simpler than the ARG but whose sampling distribution agrees with \eref{eq:zerothorder} and \eref{eq:firstorder}. We first provide an informal description. Let $\cC^{(\gr)}_{a,b,c}(t)$ denote the standard, neutral, two-locus coalescent process a time $t$ back from a sample taken at time $t=0$, with $a$, $b$, and $c$ counting the three types of sample as defined in \sref{sec:notation}. 
Recombination occurs at the usual rate of $\gr c/2$, where $\gr = 2Nr$. Lineages ancestral to the three types 
are sometimes referred to as representing \emph{left half-fragments}, \emph{right half-fragments}, and \emph{full fragments}, respectively. Our strategy is to define a coupling on a joint probability space for the pair of processes $(\cC^{(\gr)} = (\cC^{(\gr)}_{a,b,c}(t):t \geq 0), \cD^{(\infty)} = (\cD^{(\infty)}_{a,b,c}(t)): t\geq 0))$, where $\cD^{(\infty)}$ is a simple process closely related to $\cC^{(\infty)}$ and defined below. $\cC^{(\gr)}(\go )$ is said to be coupled to $\cD^{(\infty)}(\go )$ if the two realizations have the same marginal coalescent tree at locus A and the same marginal coalescent tree at locus B. Since it is the marginal trees which govern the mutation process at each locus, coupled processes therefore have the same sampling distribution. (There should be no ambiguity arising from the fact that our coupling is not on pairs of realizations but on pairs of equivalence classes, where an equivalence class of $\cC^{(\gr)}$ or of $\cD^{(\infty)}$ is a set of realizations with the same marginal tree at locus A and the same marginal tree at locus B.)

A complete description of a coalescent process is one taking values in partitions of $[n]$, as introduced by \citet{kin:1982:SPA}, with natural extensions to incorporate recombination. We opt instead to represent $\cC^{(\gr)}$ only by its \emph{ancestral} process; that is, as a birth-death process on the \emph{number} of each type of lineage. Such a process is studied in depth by \citet{eth:gri:1990} and \citet{gri:1991}. In what follows it is understood implicitly that for any given realization of the ancestral process one could reconstruct a complete coalescent process---an ARG---given some additional independent randomness. Provided the ancestral processes of $\cC^{(\gr)}$ and $\cD^{(\infty)}$ remain coupled, then it is also always possible to couple their respective \emph{coalescent} processes. For example, a decrease by one in the ancestral process corresponds to a coalescence event in the coalescent process, which can be realized by merging two uniformly chosen blocks in the partition of $[n]$. A coupling of two \emph{ancestral} processes lets us couple the corresponding \emph{coalescent} processes if we always pick the same pair of blocks to merge in the two processes. With this kept in mind, it is sufficient for the argument developed below to consider the simpler ancestral process representation.

Recall the two-locus ancestral process for the coalescent with recombination: Going backwards in time, each pair of lineages coalesces independently at rate $1$, and each lineage ancestral at both loci recombines at rate $\gr/2$. When two lineages coalesce, they are replaced with a single lineage, and this lineage is ancestral at a given locus if either of its two progenitors were ancestral at this locus. Thus for example, with $a$, $b$, and $c$ defined as above the total rate of coalescence involving one left-half fragment and one right-half fragment is $ab$, resulting in a transition of the form $(a,b,c) \mapsto (a-1,b-1,c+1)$. The remaining transitions are given in \tref{tab:coupling}. We can now make the following concise definition.

\begin{definition}
The ancestral process $\cC^{(\gr)} = (\cC^{(\gr)}_{a,b,c}(t):t \geq 0)$ is a continuous-time Markov process on $\bbN^4$ such that $\cC^{(\gr)}_{a,b,c}(0) = (a,b,c,c)$ a.s., and with infinitesimal generator
\begin{align}
\sL f(a,b,c,c) = {} & \frac{\gr c}{2}f(a+1,b+1,c-1,c-1) + \binom{c}{2}f(a,b,c-1,c-1) \notag\\
& {}+ R_{a,b,c,c}\sG f(a,b,c,c) - \left[\frac{\gr c}{2} + \binom{c}{2} + R_{a,b,c,c}\right]f(a,b,c,c), \label{eq:genrho}
\end{align}
where 
\begin{align*}
R_{a,b,c,d} = {}& ab + ac + bd + \binom{a}{2} + \binom{b}{2},\\
\sG f(a,b,c,d) = {}& \frac{1}{2R_{a,b,c,d}}[2abf(a-1,b-1,c+1,d+1) \\
& {}+ a(a+2c-1)f(a-1,b,c,d) + b(b+2d-1)f(a,b-1,c,d)],
\end{align*}
and $f: \bbN^4 \to \bbR$ is an appropriate test function.
\end{definition}
Regard the third and fourth entries in $f$ as the number of left- and right- halves of full fragments; these entries are always equal. This representation is seemingly redundant, but it will make the coupling with the corresponding process $\cD^{(\infty)}$ (for which we allow $c\neq d$) transparent. We will define $\cD^{(\infty)}$ via the following recipe. First, take $\cC^{(\gr)}$ and let $\gr\to\infty$. Ordinarily, $\cC^{(\infty)}_{a,b,c}(0)$ moves instantaneously to the state $\cC^{(\infty)}_{a+c,b+c,0}(0+)$ and evolves thereafter according to $\sL f(a+c,b+c,0,0)$. However, our second step is to make a notational change: we reuse the third and fourth entries of $f$ by separately tracking the half-fragment lineages that \emph{originated} as full fragments: we write it as a process initiated at $(a,b,c,c)$ and evolving according to the generator
\begin{multline}
\sL^{(\infty)} f(a,b,c,d) = \binom{c}{2}f(a,b,c-1,d) + \binom{d}{2}f(a,b,c,d-1)\\
{}+ R_{a,b,c,d}\sG f(a,b,c,d) - \left[\binom{c}{2} + \binom{d}{2} + R_{a,b,c,d}\right]f(a,b,c,d).\label{eq:geninf}
\end{multline}
Third, we 
introduce an \emph{artificial} recombination process which induces transitions of the form $(a,b,c,c) \mapsto (a+1,b+1,c-1,c-1)$ at rate $\gr c/2$. This does not reflect any concrete evolutionary dynamic but merely acts as a mathematical device to facilitate a coupling between the two processes. (As a minor technical detail, we should like to allow the process ultimately to reach a state of the form $(a,b,0,0)$. We therefore make a minor adjustment, below, to this artificial process to allow for it to act even if one of $c$ or $d$ is $0$.) We therefore have the following definition.

\begin{definition}
The ancestral process $\cD^{(\infty)} = (\cD^{(\infty)}_{a,b,c}(t):t \geq 0)$ is a continuous-time Markov process on $\bbN^4$ such that $\cD^{(\infty)}_{a,b,c}(0) = (a,b,c,c)$ a.s., and with infinitesimal generator
\begin{multline}
\sH^{(\infty)} f(a,b,c,d) := \sL^{(\infty)} f(a,b,c,d) \\
{}+ \frac{\gr \max\{c,d\}}{2}[f(a+\bbI\{c > 0\},b+\bbI\{d > 0\},c-\bbI\{c > 0\},d-\bbI\{d > 0\}) \\{}- f(a,b,c,d)],\label{eq:geninf2}
\end{multline}
where $f: \bbN^4 \to \bbR$ is an appropriate test function.
\end{definition}
Transitions of this process are also summarized in \tref{tab:coupling}, and henceforth we will refer to the numberings of each type of transition given in the table. It is important to keep in mind that although $\gr$ appears as a parameter in \eqref{eq:geninf2}, the process $\cD^{(\infty)}$ acts as if the two loci are independent. The process with rate depending on $\gr$ is simply an artificial relabelling of lineages. A key observation is that this artificial process does not affect the distribution of the marginal coalescent trees, so $\cC^{(\infty)}$ and $\cD^{(\infty)}$ have the same sampling distribution.

\begin{table}[t]
\begin{center}
\caption{\label{tab:coupling}Transition rates of events in the two ancestral processes $\cC^{(\gr)}$ and $\cD^{(\infty)}$.} 

\begin{tabular}{cc|cc}
\hline
& Transition & \multicolumn{2}{c}{Rate}\\
Type & $(a,b,c,d) \mapsto $& $\cC^{(\gr)}$ 
& $\cD^{(\infty)}$\\
\hline
\i & $(a,b,c-1,d-1)$ & $c(c-1)/2^*$ 
& $0$\\
\ii & $(a,b,c-1,d)$ & $0$ 
& $c(c-1)/2$ \\
\iii & $(a,b,c,d-1)$ & $0$ 
& $d(d-1)/2$ \\
\iv & $(a-1,b,c,d)$ & $a(a+2c-1)/2$ 
& $a(a+2c-1)/2$ \\
\v & $(a,b-1,c,d)$ & $b(b+2d-1)/2$ 
& $b(b+2d-1)/2$ \\
\vi & $(a-1,b-1,c+1,d+1)$ & $ab$ 
& $ab$\\
\vii & $(a+\bbI\{c > 0\},b+\bbI\{d > 0\},$~~ \\ & ~~$c-\bbI\{c > 0\},d-\bbI\{d > 0\})$ & $\gr c/2^*$ 
& $\gr \max\{c,d\}/2$\\
\hline
\end{tabular}
\end{center}
$^*${\small Defined only when $c = d$.}
\end{table}

To summarize, we have defined two Markov processes on $\bbN^4$, $\cC^{(\gr})$ and $\cD^{(\infty)}$, 
which describe two-locus ancestral processes going backwards in time and with respective generators $\sL$ and $\sH^{(\infty)}$. $\sL$ is the generator of a standard process with recombination parameter $\gr$. $\sH^{(\infty)}$ is the generator of a standard process with recombination parameter $\infty$ and with the additional properties that left half-fragments are recorded in two categories (of multiplicity $a$ and $c$), right half-fragments are recorded in two categories (of multiplicity $b$ and $d$), and there is an artificial movement of pairs from the latter to the former as if they were still full fragments. This somewhat contrived definition has an important advantage: it is a simple matter to attempt to couple the two processes by matching each kind of event in the two generators whenever possible. A recombination event in $\cC^{(\gr)}_{a,b,c}(t)$ can be matched by an artificial recombination event in $\cD^{(\infty)}_{a,b,c}(t)$, a coalescence of type \iv{} 
in $\cC^{(\gr)}_{a,b,c}(t)$ can be matched by a coalescence of type \iv{} in $\cD^{(\infty)}_{a,b,c}(t)$, and so on.

The aforementioned description is a probabilistic coupling, which may or may not succeed since not all events can be paired off in this way. Comparing \eref{eq:genrho} and \eref{eq:geninf2}, we see that a coupling will fail if there is a type \i{} transition 
in $\cC^{(\gr)}$ or if there is a type \ii{} or type \iii{} transition 
in $\cD^{(\infty)}$. 
Define the failure times
\begin{align*}
T^{(1)}_{a,b,c} &:= \inf\{t \geq 0: \cC^{(\gr)}_{a,b,c}(t) = \cC^{(\gr)}_{a,b,c}(t-) - (0,0,1,1)\},\\
T^{(2)}_{a,b,c} &:= \inf\{t \geq 0: \cD^{(\infty)}_{a,b,c}(t) = \cD^{(\infty)}_{a,b,c}(t-) - (0,0,1,0)\},\\
T^{(3)}_{a,b,c} &:= \inf\{t \geq 0: \cD^{(\infty)}_{a,b,c}(t) = \cD^{(\infty)}_{a,b,c}(t-) - (0,0,0,1)\},
\end{align*}
and
\begin{multline*}
T^{\text{\tiny MRCA}}_{a,b,c} := \inf\Big\{t \geq 0: \cC^{(\gr)}_{a,b,c}(s) = \cD^{(\infty)}_{a,b,c}(s) \quad\forall s \leq t, \\ \cC^{(\gr)}_{a,b,c}(t) \in \{(1,1,0,0),(0,0,1,1)\}\Big\},
\end{multline*}
the first time that both loci find a most recent common ancestor in the coupled processes (with the convention $\inf \varnothing = \infty$). If $T^{\text{\tiny MRCA}}_{a,b,c} < \min\{T^{(1)}_{a,b,c}$, $T^{(2)}_{a,b,c}$, $T^{(3)}_{a,b,c}\}$, we say that the coupling has been \emph{successful}. We are now in a position to verify the observation made in \sref{sec:intro}: that we need consider whether or not a coupling has been successful only as far back as the first time that no lineages ancestral to both loci survive. For if we reach this point then, even further back in time,  jointly ancestral lineages may arise again temporarily (with $c \geq 1$), but the coupling can fail only in the unlikely [i.e.\ $O(\gr^{-2})$] event that $c \geq 2$. We formalize this argument in the following lemma.

\begin{lemma}
\label{lem:c=0}
If $c \in\{0,1\}$, the coupling between $\cC^{(\gr)}$ and $\cD^{(\infty)}$ fails with probability $O(\gr^{-2})$, as $\gr\to\infty$.
\end{lemma}
\begin{proof}
The three events causing the coupling to fail occur at rates proportional to $\binom{c}{2}$ and thus require $c \geq 2$. For the pair $(\cC^{(\gr)}_{a,b,1}, \cD^{(\infty)}_{a,b,1})$, we therefore first need to see a transition of the form $(a',b',1,1) \mapsto (a'-1,b'-1,2,2)$ for some $a',b'$, followed by one of the transitions causing the coupling to fail. Reading off the rates from the generators, each of these transitions occurs with probability $O(\gr^{-1})$. The case $c=0$ is similar, first needing a transition of the form $(a',b',0,0) \mapsto (a'-1,b'-1,1,1)$ whose probability is of $O(1)$.
\end{proof}
\begin{lemma}
\label{lem:coupling}
The coupling between $\cC^{(\gr)}$ and $\cD^{(\infty)}$ fails with the following probabilities:
\begin{equation}
\label{eq:failures}
\bbP(I^{(k)}) = \frac{1}{\gr}\binom{c}{2} + O\left(\frac{1}{\gr^2}\right) \quad \text{as }\gr\to\infty, \quad k = 1,2,3,
\end{equation}
where $I^{(k)} := \{T^{(k)}_{a,b,c} < T^{\text{\tiny MRCA}}_{a,b,c}\}$. Moreover, $\bbP(I^{(k_1)} \cap I^{(k_2)}) = O(\gr^{-2})$ for $k_1 \neq k_2$.
\end{lemma}
\begin{proof}
For $k=1$, by \lemmaref{lem:c=0} it is enough to show that
\[
\bbP(T^{(1)}_{a,b,c} < U^{(1)}_{a,b,c}) = \frac{1}{\gr}\binom{c}{2} + O\left(\frac{1}{\gr^2}\right),
\]
where
\begin{align*}
U^{(1)}_{a,b,c} := {}& \inf\left\{t \geq 0: 
\cC^{(\gr)}_{a,b,c}(t) \in \{(a',b',0,0) : a',b'\in \bbN\}\right\}
\end{align*}
is the first time $\cC^{(\gr)}$ reaches $c=0$. We proceed by induction on $c$; \lemmaref{lem:c=0} provides the base cases $c\in\{0,1\}$. First note that for any $c \geq 1$,
\begin{equation}
\label{eq:T1U1}
\bbP(T^{(1)}_{a,b,c} < U^{(1)}_{a,b,c}) = O\left(\frac{1}{\gr}\right),
\end{equation}
since this event requires at least one transition that is not a recombination. Reading off the relevant probabilities from \eref{eq:genrho}, we have for $c \geq 2$:
\begin{align*}
\bbP(T^{(1)}_{a,b,c} < U^{(1)}_{a,b,c}) = {}& \frac{\frac{\gr c}{2}}{\frac{\gr c}{2} + \binom{c}{2} + R_{a,b,c,c}}\cdot\bbP(T^{(1)}_{a+1,b+1,c-1} < U^{(1)}_{a+1,b+1,c-1}) \\
& {}+ \frac{ab}{\frac{\gr c}{2} + \binom{c}{2} + R_{a,b,c,c}}\cdot \bbP(T^{(1)}_{a-1,b-1,c+1} < U^{(1)}_{a-1,b-1,c+1})\\
& {}+ \frac{\binom{c}{2}}{\frac{\gr c}{2} + \binom{c}{2} + R_{a,b,c,c}}\cdot 1+ O\left(\frac{1}{\gr^{2}}\right),\\
= {}& \frac{1}{\gr}\binom{c}{2} + O\left(\frac{1}{\gr^2}\right),
\end{align*}
by the inductive hypothesis for the first term on the right and using \eref{eq:T1U1} for the second term. By considering
\begin{align*}
U^{(k)}_{a,b,c} := {}& \inf\left\{t \geq 0: 
\cD^{(\infty)}_{a,b,c}(t) \in \{(a',b',0,0) : a',b'\in \bbN\}\right\}, \quad k=2,3,
\end{align*}
the cases $k=2,3$ are similar. $\bbP(I^{(k_1)} \cap I^{(k_2)}) = O(\gr^{-2})$ also follows from the fact that this event requires at least two transitions which are not recombinations during the time that $c > 0$.
\end{proof}

Should the coupling fail, we can say much about the sequence of events prior to $U^{(k)}_{a,b,c}$. Intuitively, the probability that \emph{more than} one transition other than recombinations occurs is $O(\gr^{-2})$. To make this precise we denote by $\cS^{(k)}_{a,b,c}(t)$ the jump chain up to time $t$ of $\cC^{(\gr)}$ if $k=1$ and of $\cD^{(\infty)}$ if $k=2, 3$.
\begin{lemma}
\label{lem:skeleton}
Let $\sS_{a,b,c}$ denote the set of jump chains comprising sequences which start at $(a,b,c,c)$, end at the first entry of the form $(a',b',0,0)$, $a',b'\in\bbN$, and with all transitions corresponding to recombination events, except for possibly one transition. Then
\[
\bbP(\cS^{(k)}_{a,b,c}(U^{(k)}_{a,b,c}) \in \sS_{a,b,c} \mid I^{(k)}) = 1 - O\left(\frac{1}{\gr}\right) \quad \text{as }\gr\to\infty, \quad k=1,2,3.
\]
\end{lemma}
\begin{proof}
The non-recombination event causing $I^{(k)}$ occurs at time $T^{(k)}_{a,b,c}$. Inspection of the generators \eref{eq:genrho} and \eref{eq:geninf2} shows that any further transition other than a recombination occurs with probability $O(\gr^{-1})$ during the time that $c > 0$.
\end{proof}
Recall that our purpose is to obtain the sampling distribution for $\cC^{(\gr)}$. For successful couplings, this is easy to obtain since it is the same as that of $\cD^{(\infty)}$ and hence $\cC^{(\infty)}$; thus $\cC^{(\gr)}\mid I^{(1)\complement}$ has the same sampling distribution as $\cD^{(\infty)}\mid (I^{(2)} \cup I^{(3)})^\complement$. Even if the coupling fails, Lemmata \ref{lem:c=0} and \ref{lem:skeleton}, demonstrate that the behaviour of $\cC^{(\gr)}$ is still predictable enough to recover its sampling distribution up to $O(\gr^{-2})$. Roughly [up to $O(\gr^{-2})$], \lemmaref{lem:skeleton} says: if there is an event that causes the coupling to fail then this is the \emph{only} non-recombination event in the failing process before $U^{(k)}_{a,b,c}$; by \lemmaref{lem:c=0}, if it has not failed by $U^{(k)}_{a,b,c}$ then the coupling will not fail after $U^{(k)}_{a,b,c}$.

The following theorem is proven in \citet{jen:son:2009:G}; however, the following proof gives a coherent, \emph{process-level} explanation for the result.

\begin{theorem}
\label{thm:coupling}
Expressing the sampling distribution for $(\cC^{(\gr)}_{a,b,c}(t) : t\geq 0)$ as in \eref{eq:main}, the first two terms are given by \eref{eq:zerothorder} and \eref{eq:firstorder}.
\end{theorem}
\begin{proof}
Denote by $q_{\cC^{(\gr)}\mid I^{(1)}}(\bfa,\bfb,\bfc)$ the sampling distribution of the process $\cC^{(\gr)}\mid I^{(1)}$. By Lemmata \ref{lem:c=0} and \ref{lem:skeleton}, this sampling distribution is obtained up to $O(\gr^{-1})$ by picking a pair of full fragments at random to coalesce, with the remaining $c-1$ fragments all undergoing recombination, and subsequently running the process as $\cD^{(\infty)}_{a+c-1,b+c-1,0} (\stackrel{a.s.}{=} \cC^{(\infty)}_{a+c-1,b+c-1,0})$. Hence,
\begin{align}
q_{\cC^{(\gr)}\mid I^{(1)}}(\bfa,\bfb,\bfc) &= \si\sj \frac{\binom{c_{ij}}{2}}{\binom{c}{2}} q_{\cC^{(\infty)}}(\bfa,\bfb,\bfc - \bfe_{ij}) + O\left(\frac{1}{\gr}\right),\notag\\
&= \si\sj \frac{\binom{c_{ij}}{2}}{\binom{c}{2}} q^\A(\bfa+\bfc_\A-\bfe_i)q^\B(\bfb+\bfc_\B-\bfe_j) + O\left(\frac{1}{\gr}\right). \label{eq:qI1}
\end{align}
(We can also ignore the possibility of mutation prior to $U^{(1)}_{a,b,c}$ since, by the same argument as in \lemmaref{lem:skeleton}, a mutation occurs during this phase with probability $O(\gr^{-1})$.) Similarly,
\allowdisplaybreaks
\begin{align}
q_{\cD^{(\infty)}\mid I^{(2)}}(\bfa,\bfb,\bfc) &= \si \frac{\binom{c_{i\cdot}}{2}}{\binom{c}{2}} q_{\cC^{(\infty)}}(\bfa+\bfc_\A-\bfe_i,\bfb+\bfc_\B,\bfzero) + O\left(\frac{1}{\gr}\right),\notag\\
&= \si \frac{\binom{c_{i\cdot}}{2}}{\binom{c}{2}} q^\A(\bfa+\bfc_\A-\bfe_i)q^\B(\bfb+\bfc_\B) + O\left(\frac{1}{\gr}\right),\label{eq:qI2}\\
q_{\cD^{(\infty)}\mid I^{(3)}}(\bfa,\bfb,\bfc) &= \sj \frac{\binom{c_{\cdot j}}{2}}{\binom{c}{2}} q_{\cC^{(\infty)}}(\bfa+\bfc_\A,\bfb+\bfc_\B-\bfe_j,\bfzero) + O\left(\frac{1}{\gr}\right),\notag\\
&= \sj \frac{\binom{c_{\cdot j}}{2}}{\binom{c}{2}} q^\A(\bfa+\bfc_\A)q^\B(\bfb+\bfc_\B-\bfe_j) + O\left(\frac{1}{\gr}\right), \label{eq:qI3}
\end{align}
and so, together with \lemmaref{lem:coupling} 
and the observation that
\begin{multline*}
\bbP([I^{(2)} \cup I^{(3)}]^\complement)q_{\cD^{(\infty)}\mid (I^{(2)} \cup I^{(3)})^\complement}(\bfa,\bfb,\bfc) = q_{\cD^{(\infty)}}(\bfa,\bfb,\bfc)\\ {}- \bbP(I^{(2)})q_{\cD^{(\infty)}\mid I^{(2)}}(\bfa,\bfb,\bfc) -  \bbP(I^{(3)})q_{\cD^{(\infty)}\mid I^{(3)}}(\bfa,\bfb,\bfc) + O(\gr^{-2}),
\end{multline*}
we obtain
\begin{multline}
q_{\cD^{(\infty)}\mid (I^{(2)} \cup I^{(3)})^\complement}(\bfa,\bfb,\bfc) = \left[1 + \frac{2}{\gr}\binom{c}{2}\right]\bigg[q_{\cD^{(\infty)}}(\bfa,\bfb,\bfc)\\
{}- \frac{1}{\gr}\binom{c}{2}q_{\cD^{(\infty)}\mid I^{(2)}}(\bfa,\bfb,\bfc) -\frac{1}{\gr}\binom{c}{2}q_{\cD^{(\infty)}\mid I^{(3)}}(\bfa,\bfb,\bfc)\bigg] + O\left(\frac{1}{\gr^{2}}\right). \label{eq:qnotI2notI3}
\end{multline}
The key decomposition is then
\begin{align}
q(\bfa,\bfb,\bfc) = {}& \bbP(I^{(1)})q_{\cC^{(\gr)}\mid I^{(1)}}(\bfa,\bfb,\bfc) + \bbP(I^{(1)\complement})q_{\cC^{(\gr)}\mid I^{(1)\complement}}(\bfa,\bfb,\bfc) \notag\\
= {}&  \bbP(I^{(1)})q_{\cC^{(\gr)}\mid I^{(1)}}(\bfa,\bfb,\bfc)
{}+ \bbP(I^{(1)\complement})q_{\cD^{(\infty)}\mid (I^{(2)}\cup I^{(3)})^\complement}(\bfa,\bfb,\bfc) \label{eq:keydecomp}\\
={} & q_0(\bfa,\bfb,\bfc) + \frac{1}{\gr}q_1(\bfa,\bfb,\bfc) + O\left(\frac{1}{\gr^2}\right), \notag
\end{align}
using \eref{eq:failures}, \eref{eq:qI1}, \eref{eq:qI2}, \eref{eq:qI3}, and \eref{eq:qnotI2notI3}, with $q_0$, $q_1$ given by \eref{eq:zerothorder} and \eref{eq:firstorder}, respectively.
\end{proof}
\begin{remark}
It may be possible to use similar arguments to obtain a genealogical interpretation of the second-order term, $q_2$ in \eqref{eq:main}; for example, genealogies with \emph{two} events that cause the coupling to fail would surely contribute. However, as is clear from the expression for $q_2$ given in \citet{jen:son:2009:G, jen:son:2010:AAP}, this is not a simple endeavour and it is seems difficult to interpret some of the components of $q_2$.
\end{remark}
\subsection{A new ``loose-linkage'' coalescent process}
Equation \eref{eq:keydecomp} tells us that, up to $O(\gr^{-2})$, we can obtain the correct sampling distribution using the mixture
\[
\ga[\cC^{(\gr)}\mid I^{(1)}] + (1-\ga)[\cD^{(\infty)}\mid (I^{(2)}\cup I^{(3)})^\complement], \quad \ga = \frac{1}{\gr}\binom{c}{2},
\]
provided $\ga < 1$. The coupling used to prove \thmref{thm:coupling} demonstrates that we can \emph{define} a simple stochastic process for weakly correlated loci, $\cE^{(\gr)}$, as follows, whose sampling distribution agrees with \eref{eq:zerothorder} and \eref{eq:firstorder} up to $O(\gr^{-2})$.\\

\fbox{
\begin{minipage}[c]{0.9\textwidth}
{\bf Algorithm to simulate $\cE^{(\gr)}$, the \emph{loose-linkage coalescent}.}
\begin{enumerate}
\item With probability $\ga$, 
choose a pair uniformly at random from the $c$ full fragments to coalesce, and then choose uniformly from the chains in $\sS_{a,b,c}$ compatible with $I^{(1)}$. Such chains are some permutation of a sequence corresponding to this sole coalescence and $c-1$ recombinations.  Inter-event \emph{times} up to $U^{(1)}_{a,b,c}$ can be sampled according to the rates specified in \eref{eq:genrho}. Go to step 3.
\item Otherwise (w.p.\ $1-\ga$), sample from $\cD^{(\infty)}\mid (I^{(2)} \cup I^{(3)})^\complement$ up to time $U^{(2)}_{a,b,c}$ ($= U^{(3)}_{a,b,c}$), which can be achieved by running $\cD^{(\infty)}$ as usual according to \eref{eq:geninf2} but banning transitions of the form $(a,b,c,d) \mapsto (a,b,c-1,d)$ and $(a,b,c,d) \mapsto (a,b,c,d-1)$. (The rates of these transitions still contribute to the overall rate governing inter-event times, however.) Go to step 3.
\item Beyond time $U^{(k)}_{a,b,c}$ ($k = 1$ in the first case above and $k=2$ in the second), construct the remainder of the process independently using $(\cC^{(\infty)}(t-U^{(k)}_{a,b,c}):t\geq U^{(k)}_{a,b,c})$ (with the appropriate starting configuration) back to the first time both loci have found a most recent common ancestor.
\end{enumerate}
\end{minipage}}
\\

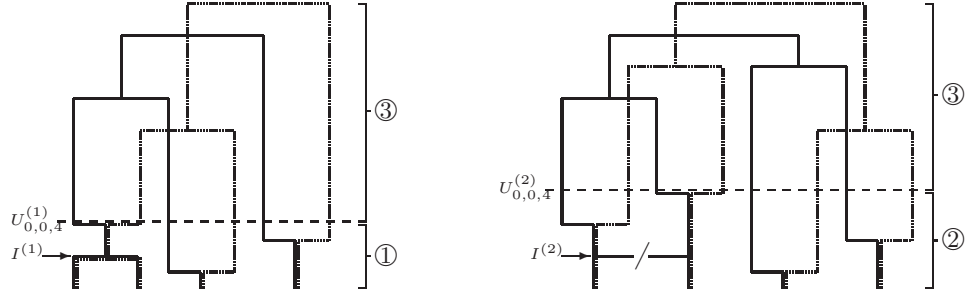
\begin{figure}[t]
\centering
\setlength{\unitlength}{0.42cm}
\begin{picture}(12,9)(-2,0)
\thicklines
\put(0,0){\line(0,1){1}}
\put(2,0){\line(0,1){1}}
\put(4,0){\line(0,1){0.5}}
\put(7,0){\line(0,1){1.5}}
\put(0,1){\line(1,0){2}}
\put(1,1){\line(0,1){1}}
\put(3,0.5){\line(1,0){1}}
\put(3,0.5){\line(0,1){5.5}}
\put(6,1.5){\line(1,0){1}}
\put(6,1.5){\line(0,1){6.5}}
\put(0,2){\line(1,0){1}}
\put(0,2){\line(0,1){4}}
\put(0,6){\line(1,0){3}}
\put(1.5,6){\line(0,1){2}}
\put(1.5,8){\line(1,0){4.5}}

\multiput(0.1,0)(0,0.1){10}{\line(0,1){0.05}}
\multiput(2.1,0)(0,0.1){10}{\line(0,1){0.05}}
\multiput(4.1,0)(0,0.1){5}{\line(0,1){0.05}}
\multiput(7.1,0)(0,0.1){15}{\line(0,1){0.05}}
\multiput(0.1,0.9)(0.1,0){20}{\line(1,0){0.05}}
\multiput(1.1,1)(0,0.1){10}{\line(0,1){0.05}}
\multiput(4.1,0.5)(0.1,0){10}{\line(1,0){0.05}}
\multiput(5.1,0.5)(0,0.1){45}{\line(0,1){0.05}}
\multiput(7.1,1.5)(0.1,0){10}{\line(1,0){0.05}}
\multiput(8.1,1.5)(0,0.1){75}{\line(0,1){0.05}}
\multiput(1.1,2)(0.1,0){10}{\line(1,0){0.05}}
\multiput(2.1,2)(0,0.1){30}{\line(0,1){0.05}}
\multiput(2.1,5)(0.1,0){30}{\line(1,0){0.05}}
\multiput(3.6,5)(0,0.1){40}{\line(0,1){0.05}}
\multiput(3.6,9)(0.1,0){45}{\line(1,0){0.05}}

\thinlines


\multiput(-0.5,2.1)(0.5,0){19}{\line(1,0){0.25}}
\put(-2,2){\tiny $U^{(1)}_{0,0,4}$}
\put(-2,0.9){\tiny $I^{(1)}$}
\put(-1,1){\vector(1,0){0.9}}
\put(9.5,5.4){\raisebox{.5pt}{\textcircled{\raisebox{-.9pt} {3}}}}
\put(9,2.1){\line(1,0){0.25}}
\put(9,9){\line(1,0){0.25}}
\put(9.25,5.6){\line(1,0){0.125}}
\put(9.25,2.1){\line(0,1){6.9}}
\put(9.5,0.8){\raisebox{.5pt}{\textcircled{\raisebox{-.9pt} {1}}}}
\put(9,0){\line(1,0){0.25}}
\put(9,2){\line(1,0){0.25}}
\put(9.25,1){\line(1,0){0.125}}
\put(9.25,0){\line(0,1){2}}

\end{picture}
\hspace{35pt}
\begin{picture}(14.5,9)(-3,0)
\thicklines
\put(0,0){\line(0,1){2}}
\put(3,0){\line(0,1){3}}
\put(6,0){\line(0,1){0.5}}
\put(9,0){\line(0,1){1.5}}
\put(2,3){\line(1,0){1}}
\put(2,3){\line(0,1){3}}
\put(5,0.5){\line(1,0){1}}
\put(5,0.5){\line(0,1){6.5}}
\put(8,1.5){\line(1,0){1}}
\put(8,1.5){\line(0,1){5.5}}
\put(5,7){\line(1,0){3}}
\put(6.5,7){\line(0,1){1}}
\put(-1,2){\line(1,0){1}}
\put(-1,2){\line(0,1){4}}
\put(-1,6){\line(1,0){3}}
\put(0.5,6){\line(0,1){2}}
\put(0.5,8){\line(1,0){6}}

\multiput(0.1,0)(0,0.1){20}{\line(0,1){0.05}}
\multiput(3.1,0)(0,0.1){30}{\line(0,1){0.05}}
\multiput(6.1,0)(0,0.1){5}{\line(0,1){0.05}}
\multiput(9.1,0)(0,0.1){15}{\line(0,1){0.05}}
\multiput(0.1,2)(0.1,0){10}{\line(1,0){0.05}}
\multiput(1.1,2)(0,0.1){50}{\line(0,1){0.05}}
\multiput(2.6,7)(0,0.1){20}{\line(0,1){0.05}}
\multiput(2.6,9)(0.1,0){60}{\line(1,0){0.05}}
\multiput(6.1,0.5)(0.1,0){10}{\line(1,0){0.05}}
\multiput(7.1,0.5)(0,0.1){45}{\line(0,1){0.05}}
\multiput(9.1,1.5)(0.1,0){10}{\line(1,0){0.05}}
\multiput(10.1,1.5)(0,0.1){35}{\line(0,1){0.05}}
\multiput(3.1,3)(0.1,0){10}{\line(1,0){0.05}}
\multiput(4.1,3)(0,0.1){40}{\line(0,1){0.05}}
\multiput(1.1,7)(0.1,0){30}{\line(1,0){0.05}}
\multiput(7.1,5)(0.1,0){30}{\line(1,0){0.05}}
\multiput(8.6,5)(0,0.1){40}{\line(0,1){0.05}}

\put(0,1){\line(1,0){1.25}}
\put(3,1){\line(-1,0){1.25}}
\put(1.32,0.9){\tiny $\cancel{\phantom{\times}}$}

\thinlines
\multiput(-1.5,3.1)(0.5,0){24}{\line(1,0){0.25}}
\put(-3,3){\tiny $U^{(2)}_{0,0,4}$}
\put(-2,0.9){\tiny $I^{(2)}$}
\put(-1,1){\vector(1,0){0.9}}
\put(11,5.9){\raisebox{.5pt}{\textcircled{\raisebox{-.9pt} {3}}}}
\put(10.5,3.1){\line(1,0){0.25}}
\put(10.5,9){\line(1,0){0.25}}
\put(10.75,6.1){\line(1,0){0.125}}
\put(10.75,3.1){\line(0,1){5.9}}
\put(11,1.3){\raisebox{.5pt}{\textcircled{\raisebox{-.9pt} {2}}}}
\put(10.5,0){\line(1,0){0.25}}
\put(10.5,3){\line(1,0){0.25}}
\put(10.75,1.5){\line(1,0){0.125}}
\put(10.75,0){\line(0,1){3}}
\end{picture}

\caption{\label{fig:looselinkage}Sampling from the loose-linkage coalescent, $\cE^{(\gr)}$, from an initial configuration $(0,0,4)$. Steps of the algorithm in the main text are denoted by circled numbers. \emph{Left}: Commence from step 1 (probability $\ga$). Step 1 samples from an approximation to \mbox{$\cC^{(\gr)}\mid I^{(1)}$} which is correct to $O(\gr^{-2})$, back as far as time $U_{0,0,4}^{(1)}$. The jump chain sampled here is $\cS^{(1)}_{0,0,4}(U^{(1)}_{0,0,4}) = ((0,0,4,4),(1,1,3,3),(1,1,2,2),(2,2,1,1),(3,3,0,0))$. Thereafter (step 3) the sample is constructed from $\cC^{(\infty)}_{3,3,0}(t-U^{(1)}_{0,0,4})$. \emph{Right}: Commence from step 2 (probability $1-\ga$). Step 2 samples from $\cD^{(\infty)}_{0,0,4}(t)\mid (I^{(2)}\cup I^{(3)})^\complement$; a transition which would cause $I^{(2)}$ is banned. Thereafter (step 3) the sample is constructed from $\cC^{(\infty)}_{4,4,0}(t-U^{(2)}_{0,0,4})$.}
\end{figure}

An example is shown in \fref{fig:looselinkage}. Simulation and inference under $\cE^{(\gr)}$ should be straightforward, since its dynamics are little more complicated than those of a coalescent process with $\gr = \infty$. Unlike our diffusion process of \sref{sec:diffusion}, it does not seem easy to write down its sampling distribution to all orders in closed-form, since that of $\cD^{(\infty)}\mid (I^{(2)}\cup I^{(3)})^\complement$ is not so obvious.

\section{Discussion}
\label{sec:discussion}
We have described two novel stochastic models of evolution for loosely linked, or weakly correlated, loci, using both diffusion- and coalescent-based arguments. As a consequence we have obtained deep insight into the simple form of the asymptotic sampling formula given by \eref{eq:zerothorder} and \eref{eq:firstorder}. Our diffusion model is based on a central limit theorem for density dependent population processes, 
which may be viewed as a separation of the timescales $N^\gb$ and $N$ (in generations), for $0 < \gb < 1$, and pioneered in population genetics by \citet{nor:1975:SIAM}. This contrasts with most research in this area, which focuses on separating the timescales $N^0 = 1$ and $N$. Indeed, both diffusion \citep{eth:nag:1980, eth:nag:1988} and coalescent \citep{moe:1998:AAP30:493, wak:2008} limits of this latter regime have been studied in detail. It is also the setting of the ``loose linkage'' limit of \citet{eth:nag:1989}. Our usage of ``loose linkage'' therefore refers to a scaling intermediate between the usual Wright-Fisher diffusion and that of \citet{eth:nag:1989}. That the pioneering approach of \citet{nor:1975:SIAM} to investigate recombination does not seem to have been considered until now supports the observation that his work is ``somewhat neglected'' \citep{wak:2005}. It would also be of interest to find a coalescent-based analogue of these results 
along the lines of \citet{moe:1998:AAP30:493}, or even a duality relationship in the manner of \citet{eth:gri:2009}.

For simplicity we have focused on a two-locus, finite-alleles, neutral model. Most of this article does not hinge heavily on these assumptions, and it should be relatively straightforward to extend our results to incorporate things like 
natural selection and more sophisticated models of mutation. 

\section*{Acknowledgments}
We gratefully acknowledge the support of the Isaac Newton Institute. Part of this work stemmed from discussions P.F.\ and Y.S.S.\ had during the 2010 program on ``Statistical Challenges Arising from Genome Resequencing.'' We also thank the generous support of the Simons Institute for the Theory of Computing. This work was completed while P.A.J.\ and Y.S.S.\ were participating in the 2014 program on ``Evolutionary Biology and the Theory of Computing.''

\bibliographystyle{imsart-nameyear}
\bibliography{master}

\end{document}